\documentclass{amsart}
\usepackage{graphicx}
\usepackage{amssymb}
\usepackage{amsfonts}
\usepackage{hyperref}
\swapnumbers
\newtheorem{theorem}{Theorem}[section]
\newtheorem{lemma}[theorem]{Lemma}
\newtheorem{corollary}[theorem]{Corollary}
\newtheorem{proposition}[theorem]{Proposition}
\theoremstyle{definition}

\newtheorem{assumption}[theorem]{Assumption}
\newtheorem{remark}[theorem]{Remark}

\numberwithin{equation}{section}
 \theoremstyle{plain}
 
 \numberwithin{equation}{section} 
 \numberwithin{figure}{section} 
 \theoremstyle{plain}
 \theoremstyle{plain}
 \theoremstyle{remark}
 \newtheorem*{acknowledgement*}{Acknowledgement}

\newcommand{\cA}{{\mathcal A}}
\newcommand{\cB}{{\mathcal B}}

\newcommand{\cD}{{\mathcal D}}

\newcommand{\cF}{{\mathcal F}}
\newcommand{\cG}{{\mathcal G}}
\newcommand{\cH}{{\mathcal H}}

\newcommand{\cL}{{\mathcal L}}
\newcommand{\cM}{{\mathcal M}}

\newcommand{\cX}{{\mathcal X}}

\newcommand{\te}{{\theta}}

\newcommand{\Om}{{\Omega}}
\newcommand{\om}{{\omega}}
\newcommand{\ve}{{\varepsilon}}
\newcommand{\del}{{\delta}}
\newcommand{\Del}{{\Delta}}
\newcommand{\gam}{{\gamma}}
\newcommand{\Gam}{{\Gamma}}

\newcommand{\sig}{{\sigma}}
\newcommand{\al}{{\alpha}}
\newcommand{\be}{{\beta}}
\newcommand{\ka}{{\kappa}}
\newcommand{\la}{{\lambda}}

\newcommand{\vp}{{\varpi}}

\newcommand{\bbN}{{\mathbb N}}

\newcommand{\bbR}{{\mathbb R}}

\newcommand{\bbZ}{{\mathbb Z}}
\newcommand{\bbI}{{\mathbb I}}
\newcommand{\bbQ}{{\mathbb Q}}

\newcommand{\brF}{{\bar F}}

\begin{document}
\title[]{Nonconventional polynomial clt}%
 \vskip 0.1cm
 \author{Yeor Hafouta and Yuri Kifer\\
\vskip 0.1cm
 Institute  of Mathematics\\
Hebrew University\\
Jerusalem, Israel}%
\address{
Institute of Mathematics, The Hebrew University, Jerusalem 91904, Israel}
\email{yeor.hafouta@mail.huji.ac.il, kifer@math.huji.ac.il}%

\thanks{ }
\subjclass[2010]{Primary: 60F17 Secondary: 60F05, 60G42, 60G15}%
\keywords{limit theorems, martingale approximation, mixing.}
\dedicatory{  }
 \date{\today}
\begin{abstract}\noindent
We obtain a functional central limit theorem (CLT) for sums of the form
$\xi_N(t)=\frac1{\sqrt N}\sum_{n=1}^{[Nt]}\big(F(X(q_1(n)),...,X(q_\ell(n)))-\brF\big)$,
where $X(n),\, n\ge0$ is a sufficiently fast mixing
vector process with some moment conditions and stationarity properties,
$F$ is a continuous function with polynomial growth and certain regularity
properties, $\bar F$ is a certain centralizing constant and $q_{i},i\geq1$
are arbitrary polynomials taking on positive integer values on positive integers, i.e. polynomials
satisfying $q_i(\bbN)\subset\bbN$, where $\bbN$ is the set of
natural numbers.
For polynomial $q_j$'s this CLT generalizes \cite{KV} which allows only
linear $q_j$'s to have the same polynomial degree.
We also  prove that  $D^2=\lim_{N\to\infty}E\xi^2_N(1)$ exists and
 provide necessary and sufficient conditions for its positivity, 
which is equivalent to the statement that  the weak limit of $\xi_N$ is not zero almost surely. 
Finally, we  study independence properties of the increments of the limiting
process. Our proofs require studying  asymptotic densities of 
special subsets of $\bbN$, which  is done in a separate section. 
As in \cite{KV}, our results
 hold true when $X_i(n)=T^nf_i$, where $T$ is a mixing subshift
of finite type, a hyperbolic diffeomorphism or an expanding transformation
taken with a Gibbs invariant measure, as well as in the case when
$X_i(n)=f_i(\Upsilon_n)$, where $\Upsilon_n$ is a Markov
chain satisfying the Doeblin condition considered as a stationary
process with respect to its invariant measure.
\end{abstract}
\maketitle
\markboth{Y. Hafouta and Y. Kifer }{Nonconventional polynomial clt}
\renewcommand{\theequation}{\arabic{section}.\arabic{equation}}
\pagenumbering{arabic}

\section{Introduction}\label{sec1}\setcounter{equation}{0}

Ergodic theorems for nonconventional averages
\[
\frac 1N\sum_{n=1}^NT^{q_1(n)}f_1\cdots T^{q_\ell(n)}f_\ell
\]
has become a well established field of research. Here $T$ is a measure preserving
transformation, $f_i$'s are bounded measurable functions and $q_i$'s are
polynomials taking on positive integer values on the positive integers, i.e. satisfying $q_i(\bbN)\subset\bbN$
for any $i=1,2,...,\ell$, where $\bbN$ is the set of natural numbers.
 The term "nonconventional" comes from \cite{Fu} and general polynomial $q_i$'s in this setup were first
considered in \cite{Be}. Taking $f_i$'s to be indicators of measurable sets
we obtain asymptotic results on numbers of multiple recurrences which was
the original motivation for this study. The probabilistic counterpart of
ergodic theorems is the law of large numbers, and from this point of view
it is natural to try to obtain other probabilistic limit theorems
for corresponding nonconventional expressions. 
This line of research started by [8] and continued in a series of paper.


In particular, a functional central limit theorem (CLT) was  obtained
in \cite{KV} for expressions of the form
\begin{equation}\label{1.1}
\xi_N(t)=\frac 1{\sqrt N}\sum_{n=1}^{[Nt]}\big(F(X(q_1(n)),...
,X(q_\ell(n)))-\brF\big)
\end{equation}
where $\{X(n), n\geq0\}$ is a sufficiently
fast mixing vector valued process with some stationarity properties
and moment conditions, $F$ is a continuous function
with polynomial growth and certain regularity properties, $\brF=\int
Fd(\mu\times\cdots\times\mu)$,
$\mu$ is the common distribution of  $X(n)$ and $q_i(n)=in$ for $1\leq i
\leq k\leq\ell$, while when $\ell\geq i>k$ they are positive functions taking
on  integer values on integers and satisfying certain growth conditions.
In the case when $q_i$'s are all polynomials those growth conditions
require that $\deg q_{i+1}>\deg q_i$ whenever $\ell>i\geq k$.
For instance, the proof from \cite{KV} does not work  for sums
of the form
\begin{eqnarray}\label{EXMP}
&\frac1{\sqrt N}\sum_{n=1}^{[Nt]}\big(F(X(n),X(n^2),X(n^2+n))-\brF\big)
\end{eqnarray}
and similar ones.

In this paper we restrict ourselves to the case of polynomial $q_i$'s but
eliminate completely the above degree growth conditions considering
arbitrary (nonconstant) polynomials taking on positive integer values on
the set $\bbN$ of positive integers and which are ordered so that $q_1(n)<q_2(n)<...
<q_\ell(n)$ for sufficiently large $n$. In particular, 
$\deg{q_{i+1}}\geq \deg{q_i}$ where equality is allowed and  some
of the differences $q_{i+1}(n)-q_i(n)$ may be (positive) constants, 
while others converge to $\infty$ as $n\to\infty$. We also recall
that the Cramer rule for linear equations $q_i(\bbN)\subset\bbN,\,
i=1,...,\ell$ implies that these polynomials must have rational coefficients.
The main goal of this paper  is to derive a functional CLT
for nonconventional expressions of the form (\ref{1.1}), where $q_i$, $i=1,...,\ell$
are general polynomials described above, $\bar F$ is the same as in (\ref{1.1})
if $q_{i+1}(n)-q_i(n)\to\infty$ as $n\to\infty$ for all $i=1,...,\ell-1$, while
when some of these differences are constants then $\bar F$ has a different
form described in the next section. As part of our proof we show that
$D^2=\lim_{N\to\infty}E\xi^2_N(1)$ exists.

We observe that \cite{KV} allows more general than polynomial nonlinear indexes $q_i(n)$
for $i>k$ only because the growth conditions on these indexes there imply that the
corresponding
limiting covariances are zero, which requires only some estimates. Here we are in the
situation where we have to ensure existence of limiting covariances which are not zero,
in general, which requires precise knowledge of the algebraic form of indexes. In short,
ensuring zero limits one needs only some estimates, while nonzero limits require more
precise knowledge of the indexes $q_1(n),...,q_\ell(n)$
 under consideration, which  leads us  to certain number
theory questions concerning  polynomials that  are resolved in Section \ref{sec4}.

After resolving  the  limiting covariances question,  
we adapt to our situation
 the  martingale approximation technique developed in \cite{KV},
and deduce the appropriate  CLT. The special
difficulty arises from the possibility of stretches of $m>1$ nonlinear
polynomials $q_i(n),...,q_{i+m-1}(n)$ of equal degree, which was not allowed
in \cite{KV}, and we overcome this difficulty relying on the number theory results from Section 
\ref{sec4}.

As soon as a  CLT is proved, it is natural to obtain conditions for
positivity of the limiting variance $D^2$, since $D^2=0$
only means that $\xi_N(1)$ converges to $0$ in the $L^2$ sense
which is less interesting than a ``true'' CLT in which $D^2>0$.
Moreover, when $D^2>0$ it becomes meaningful
to establish convergence rates in the CLT (i.e. Berry-Esseen
type estimates) and to prove a central local limit theorem.
This positivity question was not addressed in \cite{KV}. In \cite{HK} we
resolved this question  in the setup of \cite{KV}, and here we
resolve it in the polynomial setup of this paper.
Some of our conditions  are new even in the setup
of \cite{KV}, where the positivity question is nontrivial only when
$k=\ell$, which is a particular case of our setup here.

Relying on the algebraic structure of the family of polynomials $\{ q_1,...,q_\ell\}$
we  study finer properties of the weak limit $\eta$ of $\xi_N$ as $N\to\infty$.
The process $\eta$ turns out to be Gaussian but as a counterexample from \cite{KV} shows
it may have dependent increments. Still, under some  algebraic conditions, we show that the 
 increments of $\eta$ are independent on a broad
family of time intervals. 
Moreover,
under certain conditions $\eta$ turns out to be a process with stationary and 
independent increments.

As in \cite{KV}  our results hold true when, for instance,
$X(n)=T^nf$ where $f=(f_1,...,f_d)$, $T$ is a mixing subshift
of finite type, a hyperbolic diffeomorphism or an expanding transformation
taken with a Gibbs invariant measure, as well as in the case when
$X(n)=f(\Upsilon_n), f=(f_1,...,f_d)$ where $\Upsilon_n$
is a Markov chain satisfying the Doeblin condition considered as a
stationary process with respect to its invariant measure. In the dynamical
systems case each $f_i$ should be either H\" older
continuous or piecewise constant on elements of Markov partitions.
As an application we can consider $F(x_1,...,x_\ell)=x_1^{(1)}\cdots
 x_\ell^{(\ell)}$,
$x_j=(x_j^{(1)},...,x_j^{(\ell)})$, $X(n)=(X_1(n)),...,X_\ell(n))$,
$X_j(n)=\bbI_{A_j}(T^n x)$ in the dynamical systems case and
$X_j(n)=\bbI_{A_j}(\Upsilon_n)$ in the Markov chain case where
$\bbI_{A}$ is the indicator of a set $A$. Let $N(n)$ be the number of $l$'s
between $0$ and $n$ for which $T^{q_{j}(l)}x\in A_j$ for $j=0,1,...,\ell$
(or $\Upsilon_{q_{j}(l)}\in A_j$ in the Markov chains case), where we set $q_{0}=0$, namely
the number of $\ell-$tuples of return times to
$A_j$'s (either by $T^{q_j(l)}$ or by $\Upsilon_{q_j(l)}$). Then
our result yields a functional central limit theorem for the number $N([tn])$.
For some other applications of nonconventional limit theorems we refer the
reader to \cite{KV}.

\section{Preliminaries and main results}\label{sec2}\setcounter{equation}{0}

Our setup consists of a $\wp$-dimensional stochastic process $\{X(n),n\geq0\}$
on a probability space $(\Omega,\cF,P)$ and a nested family
of $\sig-algebras$ $\cF_{k,l}$, $-\infty\leq k\leq l\leq\infty$
such that $\cF_{k,l}\subset\cF_{k',l'}$ if $k'\leq k$
and $l'\geq l$. We measure the dependence between two sub $\sig-algebras$
$\cG,\cH\subset\cF$ via the quantities
\begin{equation}\label{2.1}
\varpi_{q,p}(\cG,\cH)=\sup\{\|E[g|\cG]-E[g]\|_p\,:  g\in L^q(\Om,\cH,P)\,
\,\mbox{and}\,\, \|g\|_q\leq 1\}.
\end{equation}
Then more familiar mixing (dependence) coefficients can be expressed
via the formulas  (see \cite{Br}, Ch. 4),
\begin{eqnarray*}
&\al(\cG,\cH)=\frac14\varpi_{\infty,1}(\cG,\cH),\,\,\rho(\cG,\cH)=
\varpi_{2,2}(\cG,\cH),\\
&\phi(\cG,\cH)=\frac12\varpi_{\infty,\infty}(\cG,\cH)\,\,\,\mbox{and}\,\,\,
\psi(\cG,\cH)=\varpi_{1,\infty}(\cG,\cH).
\end{eqnarray*}
 We also set
\begin{equation}\label{2.2}
\varpi_{q,p}(n)=\sup_{k\geq0}\vp_{q,p}(\cF_{-\infty,k},\cF_{k+n,\infty})
\end{equation}
and accordingly
\begin{equation*}
\al(n)=\frac14\varpi_{\infty,1}(n),\,\,\rho(n)=\varpi_{2,2}(n),
\,\,
\phi(n)=\frac12\varpi_{\infty,\infty}(n),\,\,\psi(n)=\varpi_{1,
\infty}(n).
\end{equation*}
See \cite{Br} and  Section 2 of  \cite{KV} for additional clarification and
 relations between the quantities from (\ref{2.2}). 

 In order to ensure some applications,
 in particular, to dynamical systems we do not assume that $X(n)$
is measurable with respect to $\cF_{n,n}$ but instead impose conditions
on the approximation rate
\begin{equation}\label{2.3}
\beta_q(r)=\sup_{k\geq0}\|X(k)-E(X(k)|\cF_{k-r,k+r})\|_q.
\end{equation}

Next, let $F=F(x_1,...,x_\ell)$, $x_j\in\bbR^\wp$
be a function on $(\bbR^\wp)^\ell$ such that for some $K,\iota>0$,
$\kappa\in(0,1]$ and all $x_i,z_i\in\bbR^\wp$, $i=1,...,\ell$,
we have
\begin{equation}\label{2.4}
|F(x)-F(z)|\leq K[1+\sum_{i=1}^\ell(|x_i|^\iota+|z_i|^\iota)]
\sum_{i=1}^\ell|x_j-z_j|^\kappa
\end{equation}
 and
\begin{equation}\label{2.5}
|F(x)|\leq K[1+\sum_{i=1}^\ell|x_i|^\iota]
\end{equation}
where $x=(x_1,...,x_\ell)$ and  $z=(z_1,...,z_\ell)$.

Let the nonconstant  polynomials   $q_i, i=1,...,\ell$  satisfy
$q_i(\bbN)\subset \bbN$ and  for sufficiently large $n$,
\begin{equation*}
q_1(n)<q_2(n)<...<q_\ell(n).
\end{equation*}
Then  for any $i=1,...,\ell$,
\[
\lim_{n\to\infty}(q_i(n+1)-q_i(n))>0,\,\,\lim_{n\to\infty}q_i(n)=\infty
\]
while for any $i=1,...,\ell-1$,
\begin{equation}\label{2.6}
 \deg q_{i+1}\geq\deg q_i\,\,\text{ and }\,\,
\lim_{n\to\infty}(q_{i+1}(n)-q_i(n))>0
\end{equation}
which means that these differences are either positive constants or tend to $\infty$ as 
$n\to\infty$. 
We remark that $\lim_{n\to\infty}q_i(n)=\infty$ implies that $q_i$'s have
 positive leading coefficients. Employing  Cramer's rule for solutions of
 systems of linear equations we conclude easily from
 $q_i(\bbN)\subset \bbN$ that these polynomials
 have rational coefficients. Let $0<r_1<r_2<...<r_{\hat\ell-1}<\ell$ be
 all indexes such that for $i=r_j,\, j=1,...,\hat\ell-1$ the limits in
 (\ref{2.6}) equal $\infty$ and set $r_0=0$ and $r_{\hat\ell}=\ell$.
Then $q_i-q_{r_s+1}=k_i\in\bbN$ is constant for any  $0\leq s\leq
\hat\ell-1$  and $r_s< i\leq r_{s+1}$. Let
\begin{equation}\label{2.7}
 \hat D=\{q_i-q_{r_s+1}: 0\leq s\leq \hat{\ell}-1 \,\,\mbox{and}\,\, r_s<i
 \leq r_{s+1} \}
\end{equation}
be the set of the above constant differences 
 and set $\hat\wp=
|\hat D|\wp $, where $|\Gam|$ denotes the cardinality of a finite set $\Gam$.

We do not require stationarity of the process $\{X(n), n>0\}$, 
assuming only that the distribution of $X(n)$ does not depend on $n$  and
the joint distribution of $\big(X(n_1),X(n_2),...,X(n_{2|\hat D|})\big)$ depends
only on $n_i-n_{i-1},\, i=2,3,...,2|\hat D|$ which we write for further
reference by
\begin{equation}\label{2.10}
X(n)\thicksim\mu\,\,\,\mbox{and}\,\, \big(X(n_1),X(n_2),...,X(n_k)\big)\thicksim
\mu_{n_2-n_1,n_3-n_2,...,n_k-n_{k-1}}
\end{equation}
where $k\leq2|\hat D|$ and $Y\thicksim\mu$ means that $Y$ has $\mu$ for its
distribution. 
Let $\nu_i$ be the distribution of 
$(X(0),X(q_{r_{i-1}+2}-q_{r_{i-1}+1}),...,X(q_{r_i}-q_{r_{i-1}+1}))$, i.e. 
$\nu_i=\mu_{q_{r_{i-1}+2}-q_{r_{i-1}+1},...,q_{r_i}-
q_{r_{i-1}+1}}$. If all differences $q_{i+1}(n)-q_i(n)$ tend to
 $\infty$ as $n\to\infty$ then $r_i=i$, $\nu_i=\mu$, $|\hat D|=1$ 
and the second condition in
 (\ref{2.10}) reduces to $(X(n_1),X(n_2))\thicksim\mu_{n_2-n_1}$,
 which was assumed in \cite{KV}.

For each $\te>0$, set
\begin{equation}\label{2.11}
\gam_\te^\te=\|X(n)\|_\te^\te=\int|x|^\te d\mu.
\end{equation}
Our results rely on the following assumption.
\begin{assumption}\label{ass2.1}
With $d=(\hat\ell-1)\hat\wp$ there exist $\infty>p,q\geq1$, and $\del,m>0$
with $\delta<\kappa-\frac dp$ satisfying
\begin{eqnarray}
\theta(q,p)=\sum_{n=0}^\infty\varpi_{q,p}(n)<\infty\label{2.12}\hskip1cm\\
\Lambda(q,\delta)=\sum_{r=0}^\infty\big(\beta_q(r)\big)^\del<\infty\label{2.13}\hskip1cm\\
\gamma_{m}<\infty,\gamma_{2q\iota}<\infty\,\,\,\text{ with }\,\,\,\frac12\geq
\frac1p+\frac{\iota+2}m+\frac\del q.\label{2.14}
\end{eqnarray}
\end{assumption}

To simplify  formulas we assume the centering condition
\begin{equation}\label{2.8}
\brF =\int F(y_1,...,y_{\hat\ell})d\nu _1(y_1)\cdots d\nu_{\hat\ell}
(y_{\hat\ell})=0
\end{equation}
where $y_i=(x_{r_{i-1}+1},x_{r_{i-1}+2},....,x_{r_i}),\,\,i=1,...,\hat\ell$.
Condition (\ref{2.8}) is not really a restriction since
we can always replace $F$ with $F-\brF$. It follows from Lemma 4.3 in
\cite {KV} that  $\brF$ is the limit of the expectations
$EF\big(X(q_1(n)),...,X(q_\ell(n))\big)$ as
$n\to\infty$.
Notice that if all differences $q_{i+1}(n)-q_i(n)$ tend to
$\infty$ as $n\to\infty$, then 
as in \cite{KV},
\begin{eqnarray}
\bar F=\int F(x_1,...,x_\ell)d\mu(x_1)\cdots\mu(x_\ell).
\end{eqnarray}

Our first goal  is to prove a functional central limit theorem for
\begin{equation}\label{2.9}
\xi_N(t)=\frac1{\sqrt N}\sum_{n=1}^{[Nt]} F\big(X(q_1(n)),...,
X(q_\ell(n))\big)
\end{equation}
with the function $F$ and the polynomials $q_i,\, i=1,...,\ell$ described above.

It will be convenient to represent
the function $F=F(x_1,...,x_\ell)=F(y_1,...,y_{\hat\ell})$ in the form
\begin{equation}\label{2.15}
F=F_1(y_1)+...+F_{\hat\ell}(y_1,...,y_{\hat\ell})
\end{equation}
where for $i<\hat\ell$,
\begin{eqnarray}\label{2.16}
F_i(y_1,...,y_i)=\int F(y_1,...,y_i,w_{i+1},...,w_{\hat\ell})d\nu_{i+1}(w_{i+1})
\cdots d\nu_{\hat\ell}(w_{\hat\ell})\\
-\int F(y_1,...,y_{i-1},w_i,...,w_{\hat\ell})d\nu_i(w_i)\cdots d\nu_{\hat\ell}
(w_{\hat\ell})\hskip1.cm\nonumber
\end{eqnarray}
and
\begin{equation}\label{2.16+}
F_{\hat\ell}(y_1,...,y_{\hat\ell})=F(y_1,...,y_{\hat\ell})-\int
F(y_1,...,y_{\hat\ell-1},w_{\hat\ell})d\nu_{\hat\ell}(w_{\hat\ell})
\end{equation}
which ensures that
\begin{equation}\label{2.17}
\int F_i(y_1,...,y_{i-1},w_i)d\nu_i(w_i)=0 \,\,\forall  y_1,...,y_{i-1}.
\end{equation}

Next, let $0=i_0< i_1<...<i_v=\ell $ and $m_1<m_2<...<m_v$
 be such that $\deg{q_i}=m_k$ whenever  $ i_{k-1}<i\leq i_k$.
Then  $\{i_k\}_{k=0}^v\subset \{r_k\}_{k=0}^{\hat\ell}$ and
we can  write
\begin{equation}\label{2.18}
q_i(x)=\sum_{s=0}^{m_k}a_s^{(i)}x^s\,\,\,\mbox{if}\,\,\,\, i_{k-1}<i\leq{i_k}.
\end{equation}
For any $i_{k-1}<i,j\leq i_k$ set   $c_{i,j}=\big(\frac{a_{m_k}^{(j)}}
{a_{m_k}^{(i)}}\big)^{\frac1{m_k}}>0$ which can be written also as
$c_{i,j}=\lim_{x\to\infty}\frac{q_i^{-1}(q_j(x))}x$.
Observe that for any  $1\leq s\leq\hat\ell$  there exists a unique $k$ such
that $  i_{k-1}\leq r_{s-1}< r_s\leq i_k$ and set
\begin{equation}\label{2.18+}
\xi_{s,N}(t)=\frac1{\sqrt N}\sum_{n=1}^{[Ntc_{r_s,i_{k-1}+1}]}F_s
\big(X(q_1(n)),...,X(q_{r_s}(n))\big).
\end{equation}
The following definition is important.  We say  that 
two polynomials $q$ and $p$ are \emph{equivalent} and write $q\equiv p$ if there 
exist $a,b,c\in\bbQ$ satisfying $q(y)=p(ay+b)+c$ for any $y\in\bbR$. This is 
clearly an  equivalence relation, and  we denote by $\mathcal A$ the set of all
equivalence classes. It is clear that $q\equiv p$
implies $\deg q=\deg p$, and for any $A\in\mathcal A$ let $d_A$ be the 
mutual degree of the members of $A$. We note that the class   $\cL_1$ of linear polynomials
with rational coefficients contains all the linear polynomials among $q_1,...,q_\ell$.

Next, by (\ref{2.15}) we can write
\begin{equation}\label{2.19}
\xi_N(t)=\sum_{k=1}^v\xi_N^{(k)}(t)=\sum_{A\in\mathcal A}\xi_N^{(A)}(t)
\end{equation}
where
$\xi_N^{(k)}(t)=\sum_{s:\, i_{k-1}<r_s\leq i_k}\xi_{s,N}
(c_{i_{k-1}+1,r_{s}}t)$ and 
$\xi_N^{(A)}(t)=\sum_{s: q_{r_s}\in A}\xi_{s,N}(c_{i_A,r_s}t)$, 
where $i_A=i_{k_A-1}+1$ and
 $k_A$ is such  $\deg q_j=m_k$ for any $q_j\in A$. We note that $\xi_N^{(1)}(t)=\xi_N^{(\cL_1)}(t)$
when $q_1$ is linear.

\subsection{Central limit theorem}
Our main result is the following theorem.
\begin{theorem}\label{thm2.2}\textbf{(i)}
Suppose that Assumption \ref{ass2.1} is satisfied. Then the $\hat{\ell}-$dimensional
process $\{\xi_{i,N}(t)\}_{i=1}^{\hat{\ell}}$ converges in distribution
as $N\to\infty$ to a centered $\hat\ell$-dimensional Gaussian process
$\{\eta_i(t)\}_{i=1}^{\hat\ell}$ with stationary independent increments
 and covariances having the form
\begin{equation}\label{2.19+}
E\big(\eta_i(s)\eta_j(t)\big)=\min(s,t)D_{i,j}=\lim_{N\to\infty}E\big(\xi_{i,N}(s)
\xi_{j,N}(t)\big).
\end{equation}
For any $i$ and $j$ such that $\deg{q_{r_i}}=\deg{q_{r_j}}$ the limit
$D_{i,j}$ is given by Propositions \ref{prop5.2} and \ref{prop5.3}.

\textbf{(ii)} For any $A\in\mathcal A$ set $\eta_A=\{\eta_i:q_{r_i}\in A\}$.
Let  $A,B\in\mathcal A$ be two
distinct equivalence classes.  Then $D_{i,j}=0$ if 
$q_{r_i}\in A$ and $q_{r_j}\in B$, making the vector valued processes
 $\eta_A$ and $\eta_B$  independent. 
In particular  $D_{i,j}=0$ if $\deg{q_{r_i}}\neq\deg{q_{r_j}}$ and  
the vector valued processes $\{\eta_i:\deg q_{r_i}=d\}$ are independent for different $d$'s.
Moreover, suppose that $\deg q_{r_i}>1$. Then the variance
of $\eta_i(t)$ is given by $tD_{i,i}$, where
\begin{eqnarray}\label{2.19.1}
&D_{i,i}=c_{r_i,i_{k-1}+1}\int F_i^2(y_1,...,y_i)d\nu_1(y_1)
d\nu_2(y_2)\cdot\cdot\cdot d\nu_i(y_i).
\end{eqnarray}
Here $1\leq k\leq\hat\ell$ is the unique integer 
satisfying $i_{k-1}<r_i\leq i_k$,  which means that $\deg q_{r_i}=m_k$.

\textbf{(iii)} Finally, the distribution of the process $\xi_N(\cdot)$  converges
to a Gaussian process $\eta(\cdot)$
which can be represented in the form
\begin{equation}\label{2.19++}
\eta(t)=\sum_{k=1}^v\,\sum_{s:\, i_{k_A-1}<r_s\leq i_k}\eta_s
(c_{i_{k-1}+1,r_s}t)=\sum _{A\in\mathcal A}\,\sum_{s: q_{r_s}\in A}
\eta_s(c_{i_A,r_s}t)
\end{equation}
where $i_A=i_{k_A-1}+1$ and $k=k_A$ is such that $\deg  q_i=m_k$ for any $q_i\in A$. 
The process $\eta(\cdot)$ may not have independent increments 
if there exist $s\not=s'$ such that $q_{r_s}\equiv q_{r_{s'}}$.
Moreover, $var(\eta(t))=tD^2$ where
\begin{equation}\label{2.20}
D^2=\lim_{N\to\infty}E\xi^2_N(1)=
\sum _{A\in\mathcal A} D_A^2
\end{equation}
and
\begin{equation}\label{DA2}
D_A^2=\lim_{N\to\infty}E\big(\xi_N^{(A)}(1)\big)^2=
\sum_{s:q_{r_s}\in A}c_{i_A,r_i}D_{i,i}
+2\sum_{i<j:q_{r_i},q_{r_j}\in A}c_{i_A,r_i}D_{i,j}.
\end{equation}
\end{theorem}

The strategy of the proof of Theorem \ref{thm2.2} is based on
martingale approximations of each process $\{\xi_{i,N}(t): i_{k-1}<i\leq i_k\},\,
k=1,...,v$, which are constructed after computation of the asymptotic
covariances appearing in (\ref{2.19+}).
These pose additional difficulties here in comparison to the linear situation
$q_i(n)=in$ considered in \cite{KV}, since  we allow now polynomials $q_j,\, q_{j+1}$
with the same bigger than 1 degree which was prohibited in \cite{KV} and
restricted generality there. Moreover, we allow here polynomials $q_j,\,
q_{j+1}$ which differ only by a constant so that in this case $X(q_j(n))$
and $X(q_{j+1}(n))$ are not weakly dependent even for large $n$ which was
crucial for the proof in \cite{KV}. Nevertheless, in Section \ref{sec3} we
make a reduction to the case where the latter situation is eliminated and
$q_{j+1}(n)-q_j(n)$ tends to $\infty$ as $n\to\infty$ for all $j$. The study
of covariances $E F_i\big(X(q_1(n)),...,X(q_i(n))\big)F_j\big(X(q_1(m)),...,
X(q_j(m))\big)$ when deg$q_i=$deg$q_j$ and the described above
martingale approximations construction lead to certain number
theory questions  which were considered in \cite{KV} only in the case
$\deg{q_i}=\deg{q_j}=1$.
Here we have to deal with them also for degrees higher than 1 which leads to some
(number theory) questions concerning polynomials which we resolve in
Section \ref{sec4}.

 \subsection{Positivity of $D^2$}
A crucial problem in any CLT is to specify when the limiting Gaussian distribution
is nondegenerate, i.e. it has a positive variance $D^2$, which 
by (\ref{2.20})  is equivalent to existence of $A\in\cA$ such that $D_A^2>0$. 
For any $A\in\mathcal A$ let $m=m(A)$ be the minimal natural number such that one can write 
\begin{eqnarray}\label{A-decomp}
&A\cap\{q_{r_1},...,q_{r_{\hat\ell}}\}=\bigcup_{l=1}^mA_{m,l}
\end{eqnarray} 
where for any $l$ and
 $q,p\in A_{m,l}$ there exist $z,k\in\bbZ$ such that  $q(y)=p(y-z)+k$, for any $y\in\bbR$. 
Next, for any $C\subset A$ set 
$D_C^2=\lim_{N\to\infty}N^{-1}E(\sum_{s: q_{r_s}\in C}\xi_{s,N}(c_{i_A,r_s}))^2$. 
For any $i=1,2,...,\hat\ell$ let
\begin{eqnarray}\label{VARS}
&b_{s,i}\in(\bbR^\wp)^{r_s-r_{s-1}},\, 1\leq s\leq i\,\,\,\,\text{ and }\,\,\,\,
b_i=(b_{1,i},...,b_{i,i})
\end{eqnarray}
be sets of variables.
For any $I\subset\{q_{r_1},...,q_{r_{\hat\ell}}\}$
consider the variable 
$b^{(I)}=(b_i)_{q_{r_i}\in I}$ and let the function $G_I$ be defined by
$G_I(b^{(I)})=\sum_{i: q_{r_i}\in I} F_i(b_i)$.  

\begin{theorem}\label{thm2.3}Suppose that Assumption \ref{ass2.1} is satisfied
and set $A_l=A_{m(A),l}$, $l=1,2,...,m(A)$.

\textbf{(i)} $D^2_A=0$ if and only if  $D_{A_l}^2=0$ for any $1\leq l\leq m(A)$.  
In particular,  when $q_1$ is linear then 
$D_{\cL_1}^2=\lim_{N\to\infty}E\big(\xi_N^{(1)}(1)\big)^2=0$ if 
and only if $D_{i,j}=0$ for any 
linear $q_{r_i}$ and $q_{r_j}$.

\textbf{(ii)} Suppose that $A$ consists of nonlinear polynomials. Then
there exists a family of measures $\ka_{A_l}$, $1\leq l\leq m(A)$ (which will be defined
in  Subsection (\ref{Meas})
such that 
\begin{eqnarray}
&D_{A_l}^2=\int G_{A_l}^2(b^{(A_l)})d\ka_{A_l}(b^{(A_l)})
\end{eqnarray}
for any $l$. In particular, $D_{A_l}^2$ vanishes if and only if
$G_{A_l}$ vanishes $\ka_{A_l}$-almost surely. As a consequence, $D_A^2$ vanishes
if and only if $G_A$ vanishes  $\ka_A=\prod_{l=1}^m \ka_{A_l}$-almost surely. Moreover,
let $j$ be such that for any $i\not=j$
there exist no $z,k\in\bbZ$ satisfying $q_{r_j}(y)=q_{r_i}(y-z)+k$ for any $y\in\bbR$. Then
 $D^2_A>0$, unless the  function $F_j$ vanishes  
$\nu_1\times\cdots\times\nu_{\hat\ell}$-almost surly.
\end{theorem}

Suppose that $q_1$ is linear and
let $k$ be such that $i_1=r_k$. 
Let $G=G(x_1,...,x_{i_1})=G(y_1,...,y_k)$ be a function satisfying (\ref{2.4})-(\ref{2.5}) with 
$i_1$ in place of $\ell$ 
 and (\ref{2.8}) with $k$ in place of $\hat\ell$. Set
\begin{eqnarray*}
&G_n=G\big(X(q_1(n)),...,X(q_{i_1}(n))\big)\,
 \text{ and }\,
\sig^2=\lim_{N\to\infty} \frac1N E\big(\sum_{n=1}^NG_n\big)^2
\end{eqnarray*}
which exists by Theorem \ref{thm2.2}.
For any $i=1,...,k$  let $Y^{(i)}=\{Y^{(i)}(n),  n\geq 0\}$ be independent (in general, vector) processes
such that $Y^{(i)}$ and $\{(X(n+d_{j,i}))_{j=1}^{r_i-r_{i-1}},\,n\geq0\}$
have the same distribution, where $d_{j,i}=q_{r_{i-1}+j}
-q_{r_{i-1}+1}$, which is a constant nonnegative  integer.
\begin{theorem}\label{thm2.4}
Suppose that  Assumption \ref{ass2.1} is satisfied and that
\begin{eqnarray}\label{Mix2}
&\sum_{n\geq1}n\varpi_{q,p}(n)<\infty\, \mbox{ and }\,
\sum_{r\geq 1}r\big(\beta_q(r)\big)^\del<\infty.
\end{eqnarray}

\textbf{(i)} Set $Z_n=G\big((Y^{(i)}(q_{r_{i-1}+1}(n)))_{i=1}^k\big)$ and
$\Sigma_N=\sum_{n=1}^NZ_n$. Then the limit
\begin{eqnarray*}
s^2=\lim_{N\to\infty}\frac1N\mbox{Var}\Sigma_N
\end{eqnarray*}
exists.

\textbf{(ii)} $\sig^2>0$ if and only if $s^2>0$ and the
latter conditions hold true if and only if  there exists  no representation
of the form
\begin{equation}\label{Cob}
Z_n=V_{n+1}-V_n,\,  n=0,1,2...
\end{equation}
where $\{V_{n},\,n\geq 0\}$ is a square integrable weakly (i.e. in the
wide sense) stationary process.
\end{theorem}
Applying Theorem \ref{thm2.4} with $G=F_1+...+F_k$ we obtain conditions
for positivity of $D^2_{\cL_1}$.

\subsection{The increments of $\eta$}
In Theorem \ref{thm2.2} we claim that the increments of the weak limit
$\eta$ may not be independent.
Still, on some time intervals described in the following theorem
these increments turn out to be independent.

\begin{theorem}\label{thm2.7}Suppose that Assumption \ref{ass2.1} is satisfied.
 Let $\mathcal A$ be the set of all equivalence classes of
the equivalence relation defined before Theorem \ref{thm2.2}, and write $q\equiv p$
if $q$ and $p$ lay in the same equivalent class.

\textbf{(i)} Suppose that  $a_m^{(r_i)}>a_m^{(r_j)}$ whenever $q_{r_i}\equiv q_{r_j}$, 
$i>j$ and $m=\deg{q_{r_i}}=\deg{q_{r_j}}$. Assume that there exist $i>j$ such that 
$q_{r_i}\equiv q_{r_j}$ and set
\begin{equation}\label{2.22}
C=\min\{c_{r_j,r_i}:\,i>j,\, q_{r_i}\equiv q_{r_j}\}>1
\end{equation}
where $c_{r_j,r_i}=\big(\frac{a_m^{(r_i)}}{a_m^{(r_j)}}\big)^\frac1m$ if
$\deg {q_{r_i}}=\deg{q_{r_j}}=m$.
Let $0<t_1\leq t_2\leq t_3\leq Ct_1$. Then
\begin{equation}\label{2.23}
E\big[(\eta(t_3)-\eta(t_2))\eta(t_1)\big]=\frac12(t_3-t_2)\Delta
\end{equation}
where $\Delta=D^2-\sum_{A\in\cA}\sum_{s: q_{r_s}\in A}
c_{i_A,r_s}D_{s,s}$ and $i_A$ is defined after (\ref{2.19}). 
Therefore, the increments of $\eta$ are independent when $\eta$ is reduced
to $[K,CK]$, for any $K>0$.
Furthermore, if $\Delta\not=0$ then  
for any
$0<t_2<t_3$ one  can find $0<t_0<t_1<t_2$ such that $\eta(t_3)-\eta(t_2)$ and
$\eta(t_1)-\eta(t_0)$ are not independent.

\textbf{(ii)} Suppose that for any $A\in\mathcal A$ the leading coefficients of
all $q_i\in A$ are the same.
Then $\eta$ has stationary  and independent  increments.
\end{theorem}

We remark that the situation of Theorem \ref{thm2.7}\emph{(ii)} includes the case that 
none of the polynomials $q_{r_i}, i=1,2...,\hat\ell$ are equivalent.
The following corollary follows.

\begin{corollary}\label{cor2.8}
Suppose that $\ell=2$. Then the increments of $\eta$ are independent
 if either $q_1\not\equiv q_2$ or $q_1\equiv q_2$ and 
$c_{1,2}=1$. On the other hand, when $q_1\equiv q_2$ and $c_{1,2}> 1$
then the increments of $\eta$ are independent if and only if $D_{1,2}=0$. When 
$q_1\equiv q_2$ then the asymptotic covariance $D_{1,2}$
 may or may not vanish, both when $c_{1,2}>1$ and $c_{1,2}=1$.
\end{corollary}

\begin{remark}\label{rem2.9} In fact, our proof shows that when
$\deg{q_{r_i}}=\deg{q_{r_j}}>1$ then $D_{i,j}\not=0$
 only if one can write $q_{r_j}(y)=q_{r_i}(c(y-z)+t)+s$
for some $z\in\bbZ$ and  rational $t,s,c$  satisfying $c=\frac{\al}{\be}$,
$\gcd(\al,\be)=1$ and  $t\in\{0,1....,\al-1\}$. 
It follows that $\eta_i$
and $\eta_j$ will also be independent if  one can not
find such $z,t,s$ and $c$ for $q_{r_i}$ and $q_{r_j}$. 
 Thus, some of the conditions from Theorem \ref{thm2.7} can be slightly improved
by imposing some restrictions on such $q_{r_i}$ and $q_{r_j}$.
\end{remark}

\begin{remark}\label{rem2.10}
To shorten formulas and corresponding explanations we assume that the
polynomials $q_j,\, j=1,...,\ell$ are nonconstant. In fact, the setup
allowing also constant "polynomials" can be dealt with in the same way.
Indeed, let $q_{-k}(n)= n_{-k}<q_{-k+1}(n)= n_{-k+1}<\cdots <
q_{-1}(n)= n_{-1}$ be positive integers and we are interested in
proving a functional central limit theorem for expressions of the form
\[
\xi_N(t)=\frac 1{\sqrt N}\sum_{n=1}^{[Nt]}\big(F\big(X(n_{-k}),...,X(n_{-1}),
X(q_1(n)),...,X(q_\ell(n))\big)-\brF\big)
\]
where $q_1,...,q_\ell$ are polynomials described before (\ref{2.6}) and $\brF=\brF(\om)$ is
a centralizing random variable  defined by
\[
\brF=\int F(X(n_{-k}),...,X(n_{-1}),y_1,...,y_{\hat\ell})
d\nu_1(y_1)\cdots d\nu_{\hat\ell}(y_{\hat\ell}).
\]
The first step is the representation
\[
 F(y_{-1},y_1,...,y_{\hat\ell})-
\int F(y_{-1},y_1,...,y_{\hat\ell})d\nu_1(y_1)\cdots d\nu_{\hat\ell}(y_{\hat\ell})=
\sum_{i=1}^{\hat\ell}F_i(y_{-1},y_1,...,y_i)
\]
where $y_{-1}=(x_{-k},...,x_{-1})$ and $F-i$'s are defined as in
(\ref{2.16})-(\ref{2.17})
replacing $(y_1,...,y_i)$ with $(y_{-1},y_1,...,y_i)$.

Next, our method requires to study covariances and second moments which leads to
expectations
of expressions having the form
\[
Q=G\big( X(n_{-k}),...,X(n_{-1}),X(q_1(n)),...,X(q_\ell(n)),X(q_1(m)),...,
X(q_\ell(m))\big).
\]
Set also
\[
R(x_{-k},...,x_{-1})=EG\big(x_{-k},...,x_{-1},X(q_1(n)),...,
X(q_\ell(n)),X(q_1(m)),...,X(q_\ell(m))\big)
\]
and $l(n,m)=[\frac 12(a\min(n,m)-n_{-1})]$, where $a>0$ is such that $q_i(n)\geq an$
for any $1\leq i\leq \ell$ and sufficiently large $n$. Then imposing some H\" older and
growth conditions on $G$, which will come from (\ref{2.4}) and (\ref{2.5}) in
corresponding applications, we derive from Corollary 3.6(ii) of \cite{KV} that
\[
\| E(Q|\cF_{-\infty,l(n,m)})-R(X(n_{-k}),...,X(n_{-1}))\|_2
\]
is sufficiently small when $l(n,m)$ is large, and so $|EQ-ER(X(n_{-k}),...,
X(n_{-1}))|$ is also small. This means that in all computations of expectations
and covariances we can view $X(n_{-k}),...,X(n_{-1})$ as constants (i.e.
freeze them), and so they essentially do not influence computations. We observe
that, in fact, we can consider even more general situation obtaining functional central
limit theorem for expressions of the form
\[
\xi_N(t)=\frac 1{\sqrt N}\sum_{n=1}^{[Nt]}\big(F\big(\om,
X(q_1(n)),...,X(q_\ell(n))\big)-\brF\big).
\]
Here $F(\om,x_1,...,x_\ell)$ is a random function which is either measurable with respect
to
$\cF_{-\infty,n}$ for some $n$ or it is well approximable  by conditional
expectations with respect to these $\sig$-algebras in the sense of the
approximation coefficient appearing in (\ref{2.3}) and
\[
\brF(\omega)=\int
F(\om,y_1,...,y_{\hat\ell})d\nu_1(y_1)\cdots d\nu_{\hat\ell}(y_{\hat\ell}).
\]
\end{remark}

\begin{remark}\label{rem2.11}
In \cite{KV} a functional central limit theorem was obtained also for continuous
time nonconventional expressions of the form
\[
\xi_N(t)=\frac 1{\sqrt N}\int_0^{Nt}F(X(q_1(s)),...,X(q_\ell(s)))ds.
\]
Suppose now that all $q_j$'s are polynomials satisfying $q_j(s)\to\infty$ as
$s\to\infty$ (constant "polynomials" can be treated as in Remark \ref{rem2.10}).
The first step is again the representation $F=\sum_{i=1}^{\hat\ell}F_i$ from
(\ref{2.15}) and the corresponding representations (\ref{2.19}). Similarly to
Section 6 in \cite{KV}, we see that if $\max($deg$q_i,$deg$q_j)>1$ then
\[
\lim_{N\to\infty}\frac 1N\int_0^{Nt}\int_0^{Nt}|EF_i(q_1(u)),...,X(q_i(u)))
F_j(q_1(v)),...,X(q_i(v)))|dudv=0.
\]
It follows from here that only $F_i$'s with deg$q_i=1$ play a role in the
central limit theorem for $\xi_N$, and so essentially we reduce the problem
to the setup of \cite{KV}. If, unlike \cite{KV},  some of the  differences
$q_{i+1}-q_i$ are allowed to be constants, then this additional complication can be eliminated
reducing the problem to the case when
$\lim_{t\to\infty}\big(q_{i+1}(t)-q_i(t)\big)=\infty$
for all $i\geq 1$ as described in Section \ref{sec3}.
\end{remark}

\section{Reduction to the case $\hat{\ell}=\ell$}\label{sec3}

\setcounter{equation}{0}
In this section we make a reduction to the case where
all the limits in (\ref{2.6}) equal $\infty$.
We redefine the setup as follows. Set $p_i=q_{r_{i-1}+1},\,\,i=1,...,
\hat\ell$. Then,
\[
 \lim_{n\to\infty}(p_{i+1}(n)-p_i(n))=\infty ,\,\,i=1,....,\hat\ell-1.
\]
Define $0=j_0<j_1<...<j_v=\hat\ell$ so that  $j_{k-1}<j\leq j_k$
if and only if   $\deg{p_j}=m_k$.
Then $r_{j_k}=i_k$ which implies that $c_{r_s,i_{k-1}+1}^{m_k}$ is
 the ratio of the leading coefficients of $p_{j_{k-1}+1}$ and  $p_s$
if $j_{k-1}<s\leq j_k$.
Write  $\hat D=\{0=k_1<k_2<....<k_{|\hat D|}\}$, 
where $\hat D$ is defined in (\ref{2.7}). For any $n\in\bbN$ set
\begin{equation}\label{3.1}
Z(n)=(X(n+k_1),...,X(n+k_{|\hat D|})).
\end{equation}
Then under our assumptions,   $\{Z(n),n\geq0\}$  is a
  $\hat\wp-$dimensional process satisfying Assumption \ref{ass2.1}
 with the same $\{\cF_{n,m}\}$. Furthermore, $Z(n)$ is distributed according to
$\nu=\mu_{_{k_2-k_1,...,k_{|\hat D|}-k_{|\hat D|-1}}}$ and the distribution
of each pair $\big(Z(n),\, Z(m)\big)$ depends only on $m-n$. 

Let $z=(z_1,...,z_{\hat\ell})\in (\bbR^{\hat\wp})^{\hat\ell}$, where
$z_i=(z_{i,k_j})_{j=1}^{|\hat D|}\in (\bbR^\wp)^{|\hat D|}$.
 Set
 $\hat z_i=(z_{i,0},z_{i,q_{r_{i-1}+2}-q_{r_{i-1}+1}},...,z_{i,q_{r_i}-
 q_{r_{i-1}+1}})$ and
\begin{equation}\label{3.2}
 G(z_1,...,z_{\hat\ell})=F(\hat z_1,...,\hat z_{\hat\ell}).
\end{equation}
 It is easy to see that $G$ satisfies conditions (\ref{2.4}) and (\ref{2.5})
 (see Remark 3.3  in \cite{KV}) and
the setup determined by $\hat\ell, \{p_i,\,\,i=1,...,\hat\ell\},
\hat{\wp} \mbox{ and } G$ satisfies our assumptions for the case where all the
 limits in (\ref{2.6}) with $p_i$'s in place of $q_i$'s equal $\infty$.
Observe that
\begin{eqnarray}\label{3.3}
&\hskip0.5cm\bar G=\int G(z_1,...,z_{\hat\ell})d\nu(z_1)\cdots
d\nu(z_{\hat\ell})=
\int F(\hat z_1,...,\hat z_{\hat\ell})d\nu_1(\hat z_1)\cdots
d\nu_{\hat\ell}(\hat z_{\hat\ell})=\brF, \\
&G\big(Z(p_1(n)),Z(p_2(n)),...,Z(p_{\hat\ell}(n))\big)=F\big(X(q_1(n)),....,
X(q_\ell(n))\big)\label{3.6}
\end{eqnarray}
and
\begin{equation}\label{3.4}
G_i\big(Z(p_1(n)),...,Z(p_i(n))\big)= F_i\big(X(q_1(n)),....,X(q_{r_i}(n))\big),
\,\,i=1,...,\hat\ell
\end{equation}
where $G_i$, $i=1,...,\hat\ell$ are  defined for the function $G$ as in
(\ref{2.16})-(\ref{2.17}), replacing $(y_1,...,y_i)$
with $(z_1,...,z_i)$ and $\nu_i$  with $\nu$. Furthermore, for any $i=1,
...,\hat\ell$ define $\zeta_{i,N}(t)$ with (\ref{2.18+}) replacing 
$F_i$ with $G_i$ and $X(q_{r_i}(n))$ with $Z(p_i(n))$, i.e.,
\begin{equation}\label{3.7}
\zeta_{i,N}(t)=\frac1{\sqrt N}\sum_{n=1}^{[Ntc_{r_i,i_{k-1}+1}]}G_i
\big(Z(q_1(n)),...,Z(p_i(n))\big).
\end{equation}
Then
\begin{eqnarray}\label{3.5}
\zeta_{i,N}(t)=\xi_{i,N}(t)
\end{eqnarray}
and now we can  study $\zeta_N=\sum_{i=1}^{\hat\ell}\zeta_{i,N}$ in
place of $\xi_N$.

\section{Asymptotic density of polynomial type }\label{sec4}
\setcounter{equation}{0}

Let $q_1,...,q_k$ be distinct polynomials of degree $m\geq1$ and write 
 $q_i(y)=\sum_{s=0}^ma_s^{(i)}y^s$. We assume that 
$\lim_{y\to\infty} q_i(y)=\infty $ for each $i=1,...,k$, which means that the
 leading coefficients of $q_i$'s are positive and implies that  there exists $R>0$ such that
$q_i$'s are  strictly increasing on $[R,\infty)$. 
Set $c_{i,j}=
 \big(\frac{a_m^{(j)}}{a_m^{(i)}}\big)^\frac1m$ which can  be written also as
$c_{i,j}=\lim_{y\to\infty}\frac {q_i^{-1}(q_j(y))}y$. 

Let $1\leq t_1<t_2<...<t_u\leq  k$ and set 
\begin{eqnarray}\label{Gam}
&\Gam_{t_1,...,t_u}=\{n\in\bbN : q_{t_1}(n)\in\bigcap_{i=1}^uq_{t_i}(\bbN)\}.
\end{eqnarray} 
Then $\Gam_{t_1,...t_u}$ is the set of all natural numbers $n_1$ such that there
exists a solution $(n_2,...,n_u)\in\bbN^{u-1}$ for the equation
\begin{equation}\label{EQN}
q_{t_1}(n_1)=q_{t_2}(n_2)=q_{t_3}(n_3)=...=q_{t_u}(n_u).
\end{equation}
 We are interested in the structure of the set $\Gam_{t_1,...,t_u}$ and 
in particular in showing that its asymptotic density
\begin{equation}\label{4.0}
L(t_1,...,t_u)=\lim_{r\to\infty}\frac{|\Gam_{t_1,...,t_u}\cap[1,r]|}r
\end{equation}
exists, and in providing an explicit formula for it. Here  
$|\Gam|$ denotes the cardinality of a finite set $\Gam$. 

The following notion of equivalence is crucial. 
We say that two polynomials $q$ and $p$ are \emph{0-equivalent} if there exist $a,b\in\bbQ$
such that $p(y)=q(ay+b)$ for any $y\in\bbR$. This is clearly an equivalence relation,
which  is finer than the equivalence relation defined above Theorem \ref{thm2.2}. 
We shall use the following observation. Let $1\leq t_1<t_2<...<t_u\leq  k$. Then  
the polynomials $q_{t_1},q_{t_2},...,q_{t_u}$ are $0$-equivalent if 
and only if  $c_{t_i,t_1}\in\bbQ$  for any $2\leq i\leq u$  
and there exists  $x_i\in\bbQ$, $i=2,3,...,u$ such that
\begin{equation}\label{4.1}
q_{t_i}(y)=q_{t_1}(c_{t_1,t_i}(y-x_i))\,\,\mbox{for any $y\in\bbR$}.
\end{equation}
 The following theorem is the main result of this section.

\begin{theorem}\label{thm4.1}

(i) Let $1\leq t_1<t_2<...<t_u\leq k$. Then  the limit $L=L(t_1,...,t_u)$ from \ref{4.0} exists.  
If $u=1$ then $L=1$, while $L=0$ when  $q_{t_1},...,q_{t_u}$ are not $0$-equivalent.

(ii) Suppose that $u>1$ and that $q_{t_1},...,q_{t_u}$ are $0$-equivalent. Then, up 
to a finite number of elements, 
the  set $\Gam_{t_1,...,t_u}$ defined in (\ref{Gam}) is a union of (possibly $0$)
disjoint arithmetic progressions with common difference $a=a(t_1,...,t_u)$.
As a consequence, 
\begin{eqnarray}\label{Lim-Form}
L(t_1,...,t_u)=\frac{M(t_1,...,t_u)}{a(t_1,...,t_u)}.
\end{eqnarray}
Here $a(t_1,....,t_u)=\emph{\textbf{lcm}}(\al_2,...,\al_u)$ where 
 \emph{\textbf{lcm}} denotes the least common multiple,
\begin{equation}\label{4.01}
M(t_1,...,t_u)=M\big((\al_i,\be_i,x_i)_{i=2}^u\big)=
|(W_2\times...\times W_u)\bigcap V|,
\end{equation}
\begin{eqnarray*}
&W_i=\{0,1,...,\al_i-1\}\bigcap\frac{\al_i}{\be_i}(\bbZ-x_i), \\
&V=\{(w_2,...,w_u): \exists k_2,...,k_u\in\bbN\text{ such that } w_i-w_j=k_j\al_j-k_i\al_i
\,,\,\,\forall i,j\}, 
\end{eqnarray*}
$c_{t_i,t_1}=\frac {\be_i}{\al_i}\in\bbQ $ for some coprime $\al_i,\be_i\in\bbN$
and $x_i$ satisfies (\ref{4.1}), where $i=2,3,...,u$.
\end{theorem}

The following corollary follows.

\begin{corollary}\label{cor4.1}
Set $A=\bigcup_{i=1}^kq_1^{-1}(q_i(\bbN))$. Then  $A$ has the form
$A=\{z_1<z_2<z_3<...\}$ and the limit 
\begin{eqnarray}\label{LimCor}
c=\lim_{r\to\infty}\frac{|A\cap(0,r]|}r
\end{eqnarray}
exists, is not less than $1$ and $\lim_{n\to\infty}\frac{z_n}n=c^{-1}$.
\end{corollary}
Before proving Theorem \ref{thm4.1} and Corollary \ref{cor4.1} we give two simple examples. 
Consider the situation when $q_2(x)=x^m$. Then the
asymptotic density of the set $\{n\in\bbN: \sqrt[m]{q_1(n)}\in\bbN\}$  is zero,
unless  $q_1$  has the special  form $q_1(y)=s^m(y-x_0)^m$ for some $x_0,s\in\bbQ$. In this case
the asymptotic density equals
$\frac{|\{0,1...,\al-1\}\cap s^{-1}(\bbZ+x_0s)|}\al$ where $s=\be/\al$, for some coprime 
integers $\al$ and $\be$.
 Furthermore, consider  the case that $q_1(x)=x^m$. Then Theorem \ref{thm4.1} shows that 
the asymptotic density of the set  of natural numbers $n$ 
such that $n=\sqrt[m]{q_i(l_i)}$ for some $l_i\in\bbN$  and all $2\leq i\leq k$
is zero, unless each
$q_i$ has the special form $q_i(y)=s_i^m(y-x_i)^m$ for some
rational $s_i$ and $x_i$. 
In the case that $s_i$'s and $x_i$'s are integers it  is easy to 
see without relying on Theorem \ref{thm4.1} that the asymptotic density
equals $1/\mbox{lcm}(s_2,...,s_k)$ and  Theorem \ref{thm4.1} generalizes 
this formula for the case that they are not necessarily integers.
 
\begin{proof}
We begin with the proof of Theorem \ref{thm4.1}. Let
$1\leq t_1<t_2<...<t_u\leq k$.
When $u=1$ then $\Gam_{t_1}=\bbN$ and it is clear that $L(t_1)=1$, 
and from now one we assume that $u>1$. We first need the following
result. Let $r\notin\bbQ$, $0<\ve<\frac12$ and
$x\in\bbR$. Set
\[
B(x,r,\ve)=\{l\in\bbN:\,\exists n\in \bbN\mbox{ such that }|n-rl-x|<\ve\}.
\]
Then 
\begin{equation}\label{4.2}
\limsup_{N\to\infty}\frac1N|B(x,r,\ve)\cap[1,N]|\leq 2\ve
\end{equation}
namely, the upper asymptotic density of $B(x,r,\ve)$ does not exceed
$2\ve$.
This follows from fact that $\{(rl)\mbox{ mod }1,\,\, l\in\bbN\}$
is equidistributed on $[0,1)$ and that the condition $|n-rl-x|<\ve$
implies that $(rl)\mbox{ mod }1$ lies in a union of at most two intervals
whose lengths do not exceed $\ve$. Next, let $n_1\in\bbN$.
We are interested in solving the equations
\[
q_{t_1}(n_1)=q_{t_2}(n_2)=q_{t_3}(n_3)=...=q_{t_u}(n_u)
\]
where $n_i\in\bbN,\,\,i=2,...,u$.  First,  we make a linear change of
variables by writing $n_i=c_in_1+x_i$, where
 $c_i=\big(\frac{a_m^{(t_1)}}{a_m^{(t_i)}}\big)^{\frac1m}$. 
By Taylor's expansion around $0$  of the polynomial $R_{i,x_i}(y)=q_{t_i}
(c_iy+x_i)-q_{t_1}(y)$, where $x_i$ is considered as a parameter, we obtain
 that
\begin{equation}\label{4.3}
q_{t_i}(c_in_1+x_i)-q_{t_1}(n_1)=\sum_{s=0}^{m-1}g_{i,s}(x_i)
n_1^s
\end{equation}
where
\[
g_{i,s}(x_i)=\frac1{s!}(c_i^sq_{t_i}^{(s)}(x_i)-q_{t_1}^{(s)}(0))
\]
and $f^{(s)}$ denotes the $s-$th derivative of a function $f$.
Now considering $g_{i,s}(x_i)$ as a function of $x_i$ we see that
$\lim_{x_i\to\infty}g_{i,s}(x_i)=\infty$ for any $i=2,...,u$ and $s=0,...,m-1$.
 Therefore, there exists $K_1>0$ such that a solution for the equation
 $q_{t_1}(n_1)=q_{t_i}(n_i)$ can not exist if  $x_i>K_1$.
By writing $n_1=c_i^{-1}n_i-c_i^{-1}x_i$,
repeating the above arguments and exchanging  $n_1$ and
$n_i$ we see that there exists $K_2>0$ such that a solution can not exist
 if $x_i<-K_2$.

Next, suppose that $c_i\notin\bbQ$
for some $i\in\{2,...,u\}$. Then $q_{t_1}$ and $q_{t_i}$ are not $0$-equivalent
and we need to prove that the limit $L(t_1,...,t_u)$ from  (\ref{4.0}) vanishes.
Suppose that $|x_i|\leq K_0$, where $K_0=\max(K_1,K_2)$.
Let $y_i$ be the root of the linear polynomial $g_{i,m-1}$ and let $0<\ve<
\frac12$. Assume that
\[
n_1\in\bbN \setminus B(y_i,c_i,\ve)
\]
i.e. that $|n_i-c_in_1-y_i|=|x_i-y_i|\geq\ve$ for any $n_i\in\bbN$.
Since $|x_i|\leq K_0$, 
 $g_{i,s},\,\,s=0,...,m-1$ are continuous
and $g_{i,m-1}$ is linear we deduce that there exist $C_1,C_2>0$ independent of 
$x_i$ such that 
\[
C_2\geq\max_{0\leq s \leq m-1}|g_{i,s}(x_i)|\mbox{ and }|g_{i,m-1}(x_i)|
\geq C_1\ve.
\]
It follows
that for all $n_1$ large enough there exists no solution for the equation
$q_{t_i}(n_i)=q_{t_1}(n_1)$ since the term
$g_{i,m-1}(x_i)n_1^{m-1}$ dominates the right hand side of (\ref{4.3}). Thus,
taking upper limit and applying (\ref{4.2}) we conclude that
\[
\limsup_{r\to\infty}\frac{|\Gam_{t_1,...,t_u}\cap[1,r]|}r\leq
\limsup_{r\to\infty}\frac{|\{1\leq n\leq r: q_{t_1}(n)\in
q_{t_i}(\bbN)\}|}r\leq 2\ve
\]
where $\Gam_{t_1,...,t_u}$ is defined by (\ref{Gam}). 
 Hence, letting $\ve$ to zero we obtain
$
L(t_1,...,t_u)=\lim_{r\to\infty}\frac{|\Gam_{t_1,...,t_u}\cap[1,r]|}r=0.
$

 Next, suppose that $c_i\in\bbQ$ for each $i=2,...,u$ and that
 $|x_i|\leq K_0$, $i=2,...,u$. Then $x_i=n_i-c_in_1\in\bbN-c_i\bbN\subset\bbQ$ and
\[
|[-K_0,K_0]\cap(\bbN-c_i\bbN)|<\infty
\]
which implies that each $x_i$ can take only finite number of values.
 This together with (\ref{4.3}) yields that for all $n_1$ large enough
 there exists no solution for the equation  $q_{t_i}(n_i)=q_{t_1}(n_1)$,
 unless 
\begin{eqnarray}\label{4.4-}
c_i^sq_{t_i}^{(s)}(x_i)=q_{t_1}^{(s)}(0)\text{ for any } s=0,1,....,m-1
\end{eqnarray}
since  otherwise the right hand side of (\ref{4.3})
 is not zero for large $n_1$.  By Taylor's
 expansion around $x_i$, (\ref{4.4-}) is equivalent to 
\begin{equation}\label{4.4}
q_{t_i}(y)=q_{t_1}(c_i^{-1}(y-x_i)),\forall y\in\bbR
\end{equation}
which  means that (\ref{4.1}) is satisfied and $q_{t_1},...,q_{t_u}$ are $0$-equivalent, 
taking into account that $c_i\in\bbQ$ for any $2\leq i\leq u$. 
Thus, the set $\Gam_{t_1,...,t_u}$ is finite if $q_{t_1},...,q_{t_u}$ are not
 $0$-equivalent and then clearly the limit $L(t_1,...,t_u)$ from \ref{4.0} vanishes.  
On the other hand, (\ref{4.4}) implies that $n_i=c_in_1+x_i$ solves the
equation $q_{t_i}(n_i)=q_{t_1}(n_1)$. This solution is unique 
 when $n_1$ is sufficiently large,  since $q_{t_1}$ is strictly increasing on some ray $[R,\infty)$. 
It remains to check whether $n_i\in\bbN$ when $x_i$ satisfies (\ref{4.4})
and $n_1$ is sufficiently large. Write 
$c_i=\frac{\be_i}{\al_i}$. We demand that  $c_in_1+x_i$ should be
 a natural number. Let $n_1\in\bbN$ and for any $i=2,...,u$ write 
 $n_1=k_i\al_i+w_i$ for some nonnegative integers
$k_i,w_i$ such that  $0\leq w_i<\al_i$.
 Then, $c_in_1+x_i=k_i\be_i+(\frac{w_i\be_i}{\al_i}+
 x_i)$.
This is a natural number for large enough $n_1$
if and only if $\frac{w_i\be_i}{\al_i}+
x_i$ is an integer, i.e. if and only if
\[
n_1\in\bigcap_{i=2}^u\{n\in\bbN:n\mbox{ mod }\al_i\in W_i\}=
\bigcup_{j_i\leq d_i,\, i=2,...,u}B(\om_{j_2}^{(2)},...,\om_{j_u}^{(u)})=B
\]
where the latter is a disjoint union. Here
\[
W_i=\{0,1,...,\al_i-1\}\bigcap\frac{\al_i}{\beta_i}(\bbZ-x_i)=
\{w_1^{(i)},...,w_{d_i}^{(i)}\},
\]
and
\begin{eqnarray}\label{4.5+}
&B(\om_{j_2}^{(2)},...,\om_{j_u}^{(u)})=\{n\in\bbN: n\mbox{ mod }\al_i=w_{j_i}^{(i)},i=2,...,u\}.
\end{eqnarray}
It follows that
the sets $B$ and  $\Gam_{t_1,...,t_u}$ defined by (\ref{Gam})
differ only by a finite number of elements, and 
observe that $M(t_1,...,t_u)$ defined in (\ref{4.01}) equals the
number $M$ of  nonempty sets of the form  (\ref{4.5+}).
Hence, existence of the limit $L(t_1,...,t_u)$ from (\ref{4.0}) and the equality
 $L=M(t_1,...,t_u)/a(t_1,...,t_u)$ will follow from 
\begin{equation}\label{4.5}
\lim_{r\to\infty}\frac1r|B\cap[1,r]|=\frac{M(t_1,...,t_u)}{a(t_1,...,t_u)}=\frac Ma
\end{equation}
 where $a=\textbf{lcm}(\al_2,...,\al_u)$.

Establishing (\ref{4.5}), we first claim that if $B(\om_{j_2}^{(2)},...,\om_{j_u}^{(u)})$
is not empty then there exists $n_0\in\bbN$ such that
\begin{equation}\label{form}
B(\om_{j_2}^{(2)},...,\om_{j_u}^{(u)})=\{n_0+ka\}_{k\geq 0}
\end{equation}
where $a=a(t_1,...,t_u)$.
Indeed, let $n_0$ be the smallest member of $B(\om_{j_2}^{(2)},...,\om_{j_u}^{(u)})$
and write  $n_0=k_i\al_i+w_{j_i}^{(i)}$ for some integer $k_i\geq 0$, $i=2,3,...,u$. 
Then for any integer $k\geq 0$ we have
$n_0+ka=(k_i+k\frac a{\al_i})\al_i+w_{j_i}^{(i)}$ which is clearly an element
of $B(\om_{j_2}^{(2)},...,\om_{j_u}^{(u)})$. Thus, 
$\{n_0+ka\}_{k\geq0}\subset B(\om_{j_2}^{(2)},...,\om_{j_u}^{(u)})$ and 
in particular $B(\om_{j_2}^{(2)},...,\om_{j_u}^{(u)})$ is infinite.

On the other hand, let $n\in B(\om_{j_2}^{(2)},...,\om_{j_u}^{(u)})$. Then
for any $2\leq i\leq u$ there exists an integer $k_i\geq 0$ such that $n=k_i\al_i+w_{j_i}^{(i)}$.
Let $m$ be the closest to $n$ element of $B(\om_{j_2}^{(2)},...,\om_{j_u}^{(u)})$
satisfying $m>n$. 
Such element exists since 
this set is infinite. Write
$m=(k_i+m_i)\al_i+w_{j_i}^{(i)}$ where $m_i\in\bbN$ and $2\leq i\leq u$.  
Then $\al_im_i=\al_sm_s$ for any $2\leq i,s\leq u $, which
implies that $\al_2m_2$ is divisible by $\al_2,...,\al_u$, and so
it is also divisible by $a$ and so we can write  $m=n+la$. Thus 
$ B(\om_{j_2}^{(2)},...,\om_{j_u}^{(u)})\subset\{n_0+ka\}_{k\geq0}$  and
(\ref{form}) follows.
We conclude by (\ref{form}) that 
\[
\lim_{r\to\infty}\frac 1r|B(\om_{j_2}^{(2)},...,\om_{j_u}^{(u)})\cap[1,r]|=\frac 1a
\]
 and (\ref{4.5}) follows since the sets  $B(\om_{j_2}^{(2)},...,\om_{j_u}^{(u)})$ and 
$B(\om_{j'_2}^{(2)},...,\om_{j'_u}^{(u)})$ are disjoint when 
$(\om_{j_2}^{(2)},...,\om_{j_u}^{(u)})\not=(\om_{j'_2}^{(2)},...,\om_{j'_u}^{(u)})$.
\\

Now we prove Corollary \ref{cor4.1}.  For each $1\leq i\leq k$ set 
$A_i=q_1^{-1}(q_i(\bbN))$. Then $A=\bigcup_{i=1}^kA_i$.
There exists $R>0$ such that the functions $q_i^{-1}\circ{q_j},
\,\,1\leq i,j\leq k $  are well defined on $[R,\infty)$ and are strictly
increasing there. Thus, up to a finite number of elements, $A$ is a union of 
$k$ increasing sequences and so it  has the form $\{z_1<z_2<z_3<...\}$. 
Next,
by making the change of variables  $x\to (q^{-1}_{t_1}
\circ q_1)(x)$ and  taking into account that
$\lim_{x\to\infty}\frac{q_i^{-1}(q_j(x))}x=c_{i,j}$, we deduce from (\ref{Lim-Form}) 
that for any $u$ and $1\leq t_1<t_2<...<t_u\leq k$, 
\begin{eqnarray}\label{CorLim2}
&\lim_{r\to\infty}\frac{|\bigcap_{j=1}^uA_{t_j}\cap(0,r]|}r=
c_{t_1,1}L(t_1,...,t_u)
\end{eqnarray}
where $L(t_1,...t_u)$ is the limit from (\ref{4.0}), which exists by Theorem \ref{thm4.1}.
Existence of the limit $c$ from (\ref{LimCor}) follows  by (\ref{CorLim2}) and the 
inclusion-exclusion principle. The limit $c$ is not less than one since $\bbN\subset A_1$.
Completing the proof of Corollary \ref{cor4.1}, let $m_0$ be such that $z_{m_0}\leq 0< z_{m_0+1}$
where we set $z_0=0$ if $z_1> 0$. Then $|A\cap (0,z_n]|=n-m_0$
 for any $n>m_0$. Dividing both sides
by $z_n$, taking into account that $\lim_{n\to\infty}z_n=\infty$ 
(since $\bbN\subset A_1$), yields $\lim_{n\to\infty}z_n/n=c^{-1}$. 
\end{proof}
\begin{remark}\label{rem4.2}
The question whether $M(t_1,...,t_u)>0$ is
interesting since this means that the asymptotic density is positive.
As we have the explicit formula (\ref{4.01}), provided $\al_i,\be_i$ and $x_i$
satisfying \ref{4.1} are given, this question can be resolved.
In the next section  formula (\ref{4.01}) with $u=2$ will appear in our covariance
formulas, and in this simple case it shows 
 that $M(t_1,t_2)>0$
 if and only if (\ref{4.1}) is satisfied with  $x_{t_2}=z-c_{t_2,t_1}t$
where $c_{t_2,t_1}=\be/\al$, $\gcd(\al,\be)=1$, $t\in\{0,1...,\al-1\}$
and $z\in\bbZ$. In particular when $c_{t_2,t_1}=1$ then $M(t_1,t_2)>0$
if and only if $x_{t_2}\in\bbZ$ and in this case $M(t_1,t_2)=1$.
\end{remark}

\begin{remark}\label{rem4.3}
It is possible to obtain convergence rate in the limit
(\ref{4.1}) (and thus also in the limits from Corollary \ref{cor4.1}).
When $c_{t_i,t_j}\in\bbQ$ for any $i$ and $j$ 
the proof of Theorem \ref{thm4.1} shows that the convergence
rate has the form $\frac CN$. In case that $c_{t_i,t_j}$ is not rational for some $i$ and 
$j$ the 
Erd\H{o}s-Tur\'an inequality (see Theorem 2.5 in \cite{NK}) yields
$\frac1N|B(x,c_{t_i,t_j},\ve)\cap[1,N]|\leq 2\ve+C_1\frac{\ln N}N$
for any $\ve>0$, where $C_1>0$ is an absolute constant.
Using this inequality in place of  (\ref {4.2})
and then optimizing the obtained upper bound (by taking $\ve=\ve_N$ of the form 
$\ve_N=\frac C{\sqrt N}$) yields in (\ref{4.1}) a convergence rate 
of the form $\frac C{\sqrt N}$. 

\end{remark}

\section{Limiting covariances}\label{sec5}\setcounter{equation}{0}

In this section we  prove the second equality from (\ref{2.19+}), 
assuming (\ref{2.4}), (\ref{2.5}) and Assumption \ref{ass2.1}. Then 
we show that $D^2=\lim_{N\to\infty}E\xi^2_N(1)$
exists and satisfies (\ref{2.20}). Relying on Section \ref{sec3} we deal 
from now on only with the case $\hat\ell=\ell$, i.e. 
we assume that all the limits in (\ref{2.6}) equal $\infty$.
We begin with the following observation. 
The moments conditions in Assumption \ref{ass2.1},
the definition of the functions $F_i$ from (\ref{2.16})-(\ref{2.16+})
 and (\ref{2.4})--(\ref{2.5}) yield that for any $n\in\bbN$ and $i=1,...,\ell$,
 \begin{equation}\label{5.0.0}
\|F_i\big(X(q_1(n)),...,X(q_i(n))\big)\|_2\leq 2K(1+\ell\gam_{2\iota})<\infty
 \end{equation}
 where $\gamma_{2\iota}$ is defined in (\ref{2.11}) and $K$ is from (\ref{2.4})-(\ref{2.5}).
 Next, the following lemma is  crucial. For any $1\leq i,j\leq \ell$ and $n,l\in\bbN$
 set
\begin{eqnarray*}
&b_{i,j}(n,l)=E\big[F_i\big(X(q_1(n)),...,X(q_i(n))\big)F_j\big(X(q_1(l)),...,
X(q_j(l))\big)\big].
\end{eqnarray*}
\begin{lemma}\label{lem5.1}
Suppose that all the limits in (\ref{2.6}) equal $\infty$. Then there exists a
nonincreasing function
$h(m)\geq0$, satisfying $\sum_{m=1}^\infty h(m)<\infty$,
such that for any $1\leq i,j\leq\ell$,
\[
\sup_{n,l:\,s_{i,j}(n,l)\geq m}|b_{i,j}(n,l)|\leq h(m)
\]
where $ \hat s_{i,j}(n,l)=\min(q_i(n)-q_j(l),n)\mbox{ and }s_{i,j}(n,l)
=\max(\hat s_{i,j}(n,l),\hat s_{j,i}(l,n))$.
 \end{lemma}
This assertion was proved in Lemma 4.2 of \cite{KV},  relying on the mixing
rates (\ref{2.12})--(\ref{2.13}) and on  the inequality $q_{i+1}(n)-q_i(n)
\geq n$ for any $i$ and a sufficiently large $n$.
 In our polynomial  setup there exists $C>0$ such that $q_{i+1}(n)-q_i(n)
 \geq Cn$ for any $i$ and a sufficiently large $n$, and so
 the proof of Lemma \ref{lem5.1} for our setup proceeds in the same way as
 in \cite{KV}.

Next let $j$ be such that $\deg q_j=1$. 
 Write $q_s(y)=a_1^{(s)}y+a_0^{(s)}$, for any $1\leq s\leq j$.
Observe that our assumption  that $q_s(\bbN)\subset\bbN$ in this linear case 
implies that $a_1^{(s)},a_0^{(s)}\in\bbZ$ 
and set $\nu_{s,j}=\gcd(a_1^{(s)},a_1^{(j)})$. 
 Let  $i\leq j$ and let 
$(s_k,t_k),\,\,k=1,...,r$ be  the pairs $(s,t)$, $1\leq s
\leq i$, $1\leq t \leq j$ satisfying $c_{s,t}=c_{i,j}$, where 
$c_{s,t}=a_1^{(t)}/a_1^{(s)}$. For any $u=\nu_{i,j}k\in\nu_{i,j}\bbZ$, 
let the measure $m_{i,j}^{(u)}$ be defined  by
\begin{eqnarray*}
dm_{i,j}^{(u)}(x,y)=\prod_{k=1}^rd\mu_
{c_{i,s_k}u+a_0^{(s_k)}-a_0^{(t_k)}}(x_{s_k},y_{t_k})
\times\prod_{s:\forall k\,s\neq s_k}d\mu(x_s)\times\prod_{t:\forall k\,t\neq
t_k}d\mu(y_t)
\end{eqnarray*}
where $x=(x_1,...,x_i)$, $y=(y_1,...,y_j)$ and 
the measures $\mu$ and $\mu_l,\,l\in\bbZ$ are defined in (\ref{2.10}). 
The measure $m_{i,j}^{(u)}$ is well defined since the equality 
$c_{s,t}=c_{i,j}$ implies that $n_i=a_1^{(i)}/\nu_{i,j}$ divides $a_1^{(s)}$ 
since $n_i$ and $n_j=a_1^{(j)}/\nu_{i,j}$ are coprime, which makes 
$c_{i,s_k}u$ an integer when $u\in\nu_{i,j}\bbZ$.

\begin{proposition}\label{prop5.2} Let $1\leq j\leq\hat\ell$ be such that $\deg{q_j}=1$ 
and let $i\leq j$. Then, for any $\al,\be>0$ the limit
\[
\lim_{N\to\infty}E\big(\xi_{i,N}(\al)\xi_{j,N}(\beta)\big)=D_{i,j}(\al,\be)
\]
exists and equals $\min(\al,\be)D_{i,j}$, where
\[
D_{i,j}=\frac{a_1^{(1)}\nu_{i,j}}{a_1^{(i)}a_1^{(j)}}
\sum_{u=-\infty}^{\infty}L_{i,j}(u)
\]
and this series converges absolutely. Here 
$L_{i,j}(u)=\int F_i(x)F_j(y)dm^{(u)}_{i,j}(x,y)$ for any $u=\nu_{i,j}k\in\nu_{i,j}\bbZ$,
while $L_{i,j}(u)=0$  for any other $u$.
\end{proposition}
This result was proved in Proposition 4.1 of \cite{KV} in the case
when  linear $q_j$'s satisfy $q_j(n)=jn$, relying on the mixing
rates (\ref{2.12})-(\ref{2.13}).
The proof (below) of Proposition \ref{prop5.2} goes on in a similar to \cite{KV} way
but  requires additional combinatorial arguments. In particular we use the following simple observation.  
 For any distinct polynomials $p_1,...,p_m$ there exists an injective function (i.e. permutation)
 $\sig:\{1,...m\}\to\{1,...,m\}$ and $L>0$ such that 
\begin{equation}\label{5.0-}
p_{\sig(1)}(l)<p_{\sig(2)}(l)<...<p_{\sig(m)}(l) \text{ for any } l> L.
\end{equation}

\begin{proof} Let $i,j,\al$ and $\be$ be as in the statement of this proposition.
Consider the decomposition
\begin{eqnarray}\label{5.0.1}
\quad \quad E\big(\xi_{i,N}(\al)\xi_{j,N}(\beta)\big)=
\sum_{|u|\leq |q_i([c_{i,1}\alpha N])|+|q_j([c_{j,1}\beta N])|}B(N,u)
\end{eqnarray}
where 
\begin{equation*}
B(N,u)=B_{i,j}(N,u,\al,\be)=\frac1N\sum_
{\substack{1\leq n\leq c_{i,1}\alpha N \\ 1\leq l
\leq c_{j,1}\beta N\\q_i(n)-q_j(l)=u}}b_{i,j}(n,l).
\end{equation*}
Clearly,  if $\hat u=u+a_0^{(j)}-a_0^{(i)}\not\in\nu_{i,j}\bbZ$ then
there exists  no solution
for  the equation $q_i(n)-q_j(l)=u$, and so $B(N,u)=0$. Thus, 
by Lemma \ref{lem5.1} it is sufficient to show that for any $u\in\bbZ$
such that $\hat u\in\nu_{i,j}\bbZ$, 
\begin{equation}\label{5.0.2}
\lim_{N\to\infty}B(N,u)=
\frac{a_1^{(1)}\nu_{i,j}}{a_1^{(i)}a_1^{(j)}}L_{i,j}(\hat u).
\end{equation}
Suppose that
$\hat u=u+a_0^{(j)}-a_0^{(i)}\in\nu_{i,j}\bbZ$. It is clear that 
(\ref{5.0.2}) follows  by 
\begin{equation}\label{5.0.3}
\lim_{\substack{n,l\to\infty\\q_i(n)-q_j(l)=u}}b_{i,j}(n,l)
=L_{i,j}(\hat u)
\end{equation}
and 
\begin{eqnarray}\label{5.0.4}
&\,\,\,\,\,\,\lim_{N\to\infty}\frac1N|\{(n,l):q_i(n)-q_j(l)=u,\,\,
1\leq n\leq c_{i,1}\alpha N, 1\leq l\leq c_{j,1}\beta N \}|=\\
&\frac{a_1^{(1)}\nu_{i,j}}{a_1^{(i)}a_1^{(j)}}\min(\al,\be)\nonumber.
\end{eqnarray}
Beginning with the proof of  (\ref{5.0.3}),  let $n,l\in\bbN$ be  such that $q_i(n)-q_j(l)=u$,
which means that $n=n_u(l)=\frac{a_1^{(j)}l+\hat u}{a_1^{(i)}}$. Let $1\leq s\leq i$ 
and $1\leq t\leq j$ and observe that
\begin{eqnarray}\label{5.0.5}
 &q_s(n)-q_t(l)=a_1^{(s)}n+a_0^{(s)}-a_1^{(t)}l-a_0^{(t)}=
\frac{a_1^{(j)}a_1^{(s)}-a_1^{(i)}a_1^{(t)}}{a_1^{(i)}}l+
z_{s,t}
\end{eqnarray}
where $z_{s,t}=c_{i,s}\hat u+a_0^{(s)}-a_0^{(t)}$. Consider the set
\begin{eqnarray*}
\Gam_u(l)=\{q_s(n):\,1\leq s\leq i\}\cup\{q_t(l):\,1\leq t\leq j\}
\end{eqnarray*}
where $n=n_u(l)$.  
Consider the polynomials 
 $q_t(y)$ and  $q_s(\frac{a_1^{(j)}y+\hat u}{a_1^{(i)}})=q_s(n_u(y))$,
 where $1\leq s\leq i$ and $1\leq t\leq j$, and let $w$ the number of distinct polynomials among 
them. 
 Applying (\ref{5.0-}) shows that there exists $L>0$ such that 
\[
\Gam_u(l)=\{m_1(l)< m_2(l)<...<m_w(l)\}\text{ for any } l>L
\]
and  each $m_a$, $1\leq a\leq w$ is one of the above polynomials.
By (\ref{5.0.5}), the difference  
$q_s(n)-q_t(l)=z_{s,t}$ is  constant, if $(s,t)=(s_k,t_k)$ for some $1\leq k\leq r$, 
where $n=n_u(l)$. 
For any other couple $(s,t)$ it follows from (\ref{5.0.5}) that 
$\lim_{l\to\infty}|q_s(n_u(l))-q_t(l)|=\infty$. We conclude that for any 
$0<i<w$ either 
$\lim_{l\to\infty}m_{i+1}(l)-m_i(l)=\infty$ or $m_{i+1}(l)-m_i(l)$ is a constant and
(\ref{5.0.3}) follows  by Lemma 4.3 from \cite{KV}. 

We remark that in the  terminology of  Lemma 4.3 from \cite{KV} we used a 
partition of $\Gam_u(l)$ 
into "rigid blocks"  consisting of pairs $\{q_{s_k}(n),q_{t_k}(l)\}$
 $1\leq k\leq r$, and of singletons $\{q_s(n)\}$ and $\{q_t(l)\}$ 
where $t\in\{1,...,j\}\setminus\{t_1,...,t_r\}$ and 
$s\in\{1,...,i\}\setminus\{s_1,...,s_r\}$.

Establishing (\ref{5.0.4}), the cardinality of
 the set from (\ref{5.0.4}) equals the
cardinality of  the set
\begin{eqnarray*}
&S_1(N)=\big\{1\leq l\leq \frac{a_1^{(1)}N\be}{a_1^{(j)}}: \frac{a_1^{(j)}l+
\hat u}{a_1^{(i)}}\in\{1,2,....,[\frac{N\al a_1^{(1)}}{a_1^{(i)}}]\}\big\}.
\end{eqnarray*}
 The set $S_2$ of all
natural numbers $l$ such that $\frac{a_1^{(j)}l+\hat u}{a_1^{(i)}}\in\bbN$
is an arithmetic progression with common difference
$d=\min(\bbN\cap\frac{a_1^{(i)}}{a_1^{(j)}}\bbN)=\frac{a_1^{(i)}}{\nu_{i,j}}$
having asymptotic density $d^{-1}$. Observe now that
$
S_1(N)=S_2\cap \{1,2,...,\min([\frac{a_1^{(1)}N\beta}{a_1^{(j)}}],
[\frac{N\alpha a_1^{(1)}-\hat u}{a_1^{(j)}}])\}.
$
Hence,
\begin{eqnarray*}
&\lim_{N\to\infty}\frac1N|S_1(N)|=d^{-1}\frac{a_1^{(1)}\min(\al,\be)}{a_1^{(j)}}=
\frac{a_1^{(1)}\nu_{i,j}}{a_1^{(j)}a_1^{(i)}}\min(\al,\be)
\end{eqnarray*}
and (\ref{5.0.4}) follows.
\end{proof}

Before formulating the limiting covariances results for nonlinear indexes, we shall need the 
following. Let $1\leq i,j\leq\ell$. Then $q_i\equiv q_j$  if and 
only if  $c_{i,j}\in\bbQ$  and there exists $x_{i,j}\in\bbQ$ such that
\begin{equation}\label{5.0}
q_i(y)=q_j(c_{j,i}(y-x_{i,j}))+q_i(x_{i,j})-q_j(0) \,\,\text{ for any }\,
y\in\bbR.
\end{equation}
Here  $q_i\equiv q_j$ means that $q_i$ and $q_j$ are equivalent with respect to 
the equivalence relation defined above Theorem \ref{thm2.2} and 
 $c_{i,j}$ is defined below (\ref{2.18}).

Next, suppose that $q_i\equiv q_j$ and write
 $c_{i,j}=a(i,j)/b(i,j)$, where  $a(i,j),b(i,j)\in\bbN$ and $\gcd(a(i,j),b(i,j))=1$.
Let $x_{i,j}\in\bbQ$ satisfying
(\ref{5.0}) and 
let $(s_k,t_k),\,\,k=1,...,r$ be the pairs $(s,t)$,
$1\leq s\leq i$,  $1\leq t \leq j$ satisfying  $\deg {q_s}=\deg{q_t}$,
 $c_{s,t}=c_{i,j}$ and that
\begin{eqnarray}\label{5.0+}
&q_s(y)=q_t(c_{j,i}(y-x_{i,j}))+q_s(x_{i,j})-q_t(0)\,\,
\text{ for any }\, y\in\bbR.
\end{eqnarray}
In case that $M(a(i,j),b(i,j),x_{i,j})$ (defined  by
(\ref{4.01})) is positive, 
let $m_{i,j}$ be  the measure  defined  by 
\begin{equation}\label{5.1}
dm_{i,j}(x,y)=\prod_{k=1}^r
d\mu_{q_{t_k}(0)-q_{s_k}(x_{i,j})}(x_{s_k},y_{t_k})
\times\prod_{s:\, \forall k\,s\neq s_k}d\mu(x_s)\times
\prod_{t:\,\forall k\,t\neq t_k}d\mu(y_t)
\end{equation}
where the measures $\mu$ and $\mu_l,\,l\in\bbZ$ are defined in
(\ref{2.10}), $x=(x_1,...,x_i)$ and  $y=(y_1,...,y_j)$.
As explained in Remark \ref{rem4.2}, $\,\,M(a(i,j),b(i,j),x_{i,j})>0$ if and only if
$x_{i,j}=z-c_{i,j}l$ for some $z\in\bbZ$ and $l\in\{0,1...,a(i,j)-1\}$. Plugging in 
$y=z+(|z|+1)a(i,j)$ in (\ref{5.0+}) shows that
$q_{t_k}(0)-q_{s_k}(x_{i,j})$ is an integer, and  so the  measure $m_{i,j}$ is well defined.

\begin{proposition}\label{prop5.3} (i) Let $1\leq j\leq\hat\ell$ be such that
$\deg q_j>1$ and let $1\leq i\leq j$. Then, for any $\al,\be>0$  the limit
\begin{equation}\label{5.1*}
\lim_{N\to\infty}E(\xi_{i,N}(\al)\xi_{j,N}(\beta))=D_{i,j}(\al,\be)
\end{equation}
exists and equals $\min(\al,\be)D_{i,j}$.
Moreover, $D_{i,j}=0$  if  $q_i$ and  $q_j$ are not equivalent 
with respect to the equivalence
relation defined above Theorem \ref{thm2.2}.
 In particular $D_{i,j}=0$ if $\deg q_i\not=\deg q_j$.

(ii) Suppose that $i<j$ and that $q_i$ and $q_j$ are equivalent.  Let
 $1\leq k\leq v$ be such that $i_{k-1}<i,j\leq i_k$ and 
 $\deg{q_j}=\deg{q_i}=m_k>1$, where $i_0,...,i_v$ are defined before (\ref{2.18}). 
Set $M_{i,j}=M(a(i,j),b(i,j),x_{i,j})$. 
 Then
\begin{eqnarray*}
&D_{i,j}=\frac{c_{j,i_{k-1}+1}M_{i,j}}{a(i,j)}\int F_i(x)F_j(y) dm_{i,j}(x,y)
\end{eqnarray*}
provided that $M_{i,j}>0$, and otherwise $D_{i,j}=0$.

(iii) Finally,
\begin{eqnarray}\label{5.1+}
&D_{j,j}=c_{j,i_{k-1}+1}\int  F_j^2(z_1,...,z_j)
d\mu(z_1) d\mu(z_2)\cdot\cdot\cdot d\mu(z_j).
\end{eqnarray}
\end{proposition}

\begin{proof}
Let $1\leq i\leq j\leq\ell$ be such that $\deg q_j>1$ and let  $\al,\be>0$. 
The proof that $D_{i,j}=0$  when
$\deg{q_i}\neq\deg{q_j}$ is a direct consequence of Lemma \ref{lem5.1}
and it proceeds as the proof of Proposition 4.5
in \cite{KV}. 
Suppose that $\deg q_i=\deg q_j=m$ and let $k$ be such that  $i_{k-1}<i,j\leq i_k$, 
 which means that  $m=m_k$. Set 
$c=c_{i,j}=(\frac{a_m^{(j)}}{a_m^{(i)}})^{\frac1m}\geq1$. 

Let $n,l\in\bbN$ and write $n=cl+x$, where 
$x$ is considered as a parameter.
We begin our proof with estimating the quantities $s_{i,j}(n,l)$ defined in Lemma 
(\ref{lem5.1}). 
By  Taylor's expansion around $0$ of
$q_i(cy+x)-q_j(y)$ as a function of $y$, with $x$ considered as
a parameter, we obtain that
\begin{equation}\label{5.3}
q_i(n)-q_j(l)=\sum_{u=0}^{m-1}g_u(x)l^u
\end{equation}
where
\begin{eqnarray}\label{5.3+}
g_u(x)=\frac1{u!}(c^uq_i^{(u)}(x)-q_j^{(u)}(0))
\end{eqnarray}
and  $f^{(u)}$ is the $u-$th derivative of a function $f$. 
Since 
$l=c^{-1}n-c^{-1}x$, similarly to (\ref{5.3}),
we have,
\begin{equation}\label{5.4}
q_j(l)-q_i(n)=\sum_{u=0}^{m-1}f_u(x)n^u
\end{equation}
 where
\[
f_u(x)=\frac1{u!}(c^{-u}q_j^{(u)}(-c^{-1}x)-q_i^{(u)}(0)).
\]
Since $\lim_{x\to\infty}g_u(x)=\infty$ for any $u=0,1,...,m-1$, 
 there exist $M_1>1$ and $C_1,C_2,C_3>0$ such that for any $x>M_1$ and
$n,l\in\bbN$ satisfying $n=cl+x$,
\[
q_i(n)-q_j(l)\geq C_1\max_{m>u\geq0}q_i^{(u)}(x)l^u\geq
C_2\max_{m>u\geq0} x^{m-u}l^u\geq C_3\max(n,l+x).
\]
Since
 $ n\geq l+x\geq l$ we conclude that there exists a constant $C_4>0$ such that
\begin{equation}\label{5.5}
\hat s_{i,j}(n,l)\geq C_4\max(n,l+x)\geq C_4\max(n,l).
\end{equation}
Similarly, by letting $x\to-\infty$, there exist $M_2>1$ and $C_5>0$ such that
 for any $x<-M_2$ and $n,l\in\bbN$ satisfying $n=cl+x$,
\begin{equation}\label{5.6}
\hat s_{j,i}(l,n)\geq c^{-1}C_5\max(l,n+|x|)\geq c^{-1}C_5\max(n,l).
\end{equation}

Next, for any  $N\in\bbN$ set
\begin{equation}\label{5.2}
J_1(N)=\{1,...,[N\al_i]\},\, J_2(N)=\{1,...,[N\be_j]\}\mbox{ and }
I(N)=J_1(N)-cJ_2(N)
\end{equation}
where $\al_i=\al c_{i,i_{k-1}+1}$ and $\be_j=\be c_{j,i_{k-1}+1}$.
Suppose that $c\notin\bbQ$. Then $q_i$ and $q_j$ are not equivalent and we are interested 
in proving that $D_{i,j}(\al,\be)=0$.
 Let $y_1$ be the root of the linear polynomial $g_{m-1}$. For any $0<\ve<\frac12$ set
\[
A_\ve=\{l\in\bbN:\,|n-cl-y_1|<\ve\mbox{ for some }n\in\bbN\}.
\]
 Then by (\ref{4.2}),
\begin{equation}\label{5.7}
\limsup_{N\to\infty}\frac1N|A_\ve\cap[1,N]|\leq 2\ve.
\end{equation} 
For any $l\in A_\ve$,  let $n=n_1(l)$
be the only positive  integer satisfying $|n-cl-y_1|<\ve$. 
Consider now the following decomposition
\begin{equation}\label{5.7+}
\frac1NE(\xi_{i,N}(\al)\xi_{j,N}(\beta))=\frac1N\sum_{l\in A_\ve\cap J_2(N): n_1(l)\in J_1(N)}
b_{i,j}(n_1(l),l)+
\frac1N\sum_{(n,l)\in B_{N,\ve}}b_{i,j}(n,l)
\end{equation}
where $B_{N,\ve}=(J_1(N)\times J_2(N))\cap\{(n,l): |n-cl-y_1|\geq\ve\}$.
Estimating the first sum of the above right hand side,
by (\ref{5.7}) and  (\ref{5.0.0} we have
\begin{equation}\label{5.8}
\limsup_N\frac1N\sum_{l\in A_\ve\cap J_2(N): n_1(l)\in J_1(N)}|b_{i,j}
(n_1(l),l)|\leq2\ve\be_i\cdot K^2\big(1+\ell\gam_{2q\iota}\big)^2=R\ve.
\end{equation}
Estimating the second sum of  the right side of (\ref{5.7+}), 
 let $n,l\in\bbN$  be such that $|n-cl-y_1|\geq\ve$, 
write  $ n=cl+x$ and assume that
 $|x|\leq M=\max(M_1,M_2)$. Then $|x-y_1|\geq\ve$.
By continuity of $g_s,\,s=0,...,m-1$
there exist $K_{1,\ve},K_2>0$ such that for any such $x$,
\[
K_2\geq\max_{0\leq s \leq m-1}|g_s(x)|\mbox{ and }|g_{m-1}(x)
|\geq K_{1,\ve}.
\]
Hence, if $\max(n,l)$ is sufficiently large  and $|x|\leq M$, then by
(\ref{5.3}),
\[
|q_i(n)-q_j(l)|\geq C_\ve\max(n, l)
\]
where $C_\ve>0$ is a constant which depends only on $\ve$. Thus, for such $n,l$ and $x$,
\begin{equation}\label{5.9}
s_{i,j}(n,l)\geq R_\ve\max(n,l)
\end{equation}
where $R_\ve>0$ is a constant which depends only on $\ve$. By (\ref{5.5}),
(\ref{5.6}), (\ref{5.9}) and Lemma \ref{lem5.1} there exist   $K_\ve, R_0>0$ 
such that
\begin{eqnarray}\label{5.9.1}
&\frac1N\sum_{(n,l)\in B_{N,\ve}}|b_{i,j}(n,l)|\leq \frac{R_0}N+
\frac1N\sum_{(n,l)\in B_{N,\ve}}h(K_\ve\max(n,l))\leq\\
&\frac{R_0}N+\frac1N
\sum_{n=1}^{[N\al_i]}nh(K_\ve n)+\frac 1N\sum_{l=1}^{[N\be_i]}lh(K_\ve l)\nonumber
\end{eqnarray}
where $h(t)=h([t])$. Observe that for any
nonnegative monotone decreasing sequence satisfying $\sum_{n=1}^\infty a_n
<\infty$,
\[
\limsup_{n\to\infty}(n a_n)\leq2\limsup_{n\to\infty}(\sum_{n\geq k>
\frac n2}a_k)=0.
\]
Thus, $\lim_{s\to\infty}sh(K_\ve s)=0$, and so
$
\lim_{N\to\infty}N^{-1}\sum_{(n,l)\in B_{N,\ve}}|b_{i,j}(n,l)|=0
$
by (\ref{5.9.1}). 
Finally, letting $\ve\to0$ in (\ref{5.8}) we obtain $D_{i,j}(\al,\be)=0$. 

Next, suppose that $c\in\bbQ$. Let
$n,l\in\bbN$ and $x\in\bbQ$ be such that  $n=cl+x$ and $|x|\leq M=\max(M_1,
M_2)$. Suppose that there exists $0< u<m$ such that $g_u(x)\not=0$. 
Observe that
\begin{equation}\label{5.9+}
|[-M,M]\cap(\bbN-c\bbN)|<\infty
\end{equation}
since $c$ is rational, which means that $x$ can take only a finite number of values.
Hence, by (\ref{5.3}) there exists $C_6>0$ such
that whenever $\max(n,l)$  is sufficiently large,
\[
|q_i(n)-q_j(l)|\geq C_6\max(l+|x|,\, n+|x|)
\]
 and therefore there exists $C_7>0$ such that  $\max(n,l)$  is sufficiently large,
\begin{equation}\label{5.10}
 s_{i,j}(n,l)\geq C_7\max(l+|x|,\, n+|x|).
\end{equation}
 Set
\begin{eqnarray*}
A(N)=I(N)\bigcap\{x|\, g_u(x)\neq0\mbox{ for some } 0< u<m\},
\\
\forall x\in\bbQ,\, B(N,x)=(J_1(N)\times J_2(N))\cap\{(n,l): n=cl+x\},
\end{eqnarray*}
and 
\begin{eqnarray*}
C(N)=(J_1(N)\times J_2(N))\cap\{(n,l): g_u(n-cl)=0,\forall\,\,\, 0<u<m\}
\end{eqnarray*} 
where $I(N),J_1(N)$ and $J_2(N)$ are defined in (\ref{5.2}).
Consider the decomposition 
\begin{equation}\label{5.10+}
\frac1NE(\xi_{i,N}(\al)\xi_{j,N}(\beta))=
\frac1N\sum_{x\in A(N)}\sum_{(n,l)\in B(N,x)}b_{i,j}(n,l)+
\frac1N\sum_{(n,l)\in C(N)}b_{i,j}(n,l).
\end{equation}
By (\ref{5.5}), (\ref{5.6}), (\ref{5.10}) and Lemma \ref{lem5.1}, there exist $C,C_8>0$ such that 
for any $N\in\bbN$,
\begin{equation}\label{5.11}
\frac1N\sum_{x\in A(N)}\sum_{(n,l)\in B(N,x)}|b_{i,j}(n,l)|\leq \frac CN+
\frac1N\sum_{x\in A(N)}\sum_{s=1}^\infty h(C_8(|x|+s)).
\end{equation}
 Observe that $\sum_{s=1}^
\infty h(C_8(|x|+s))$ converges
to $0$ as $|x|\to\infty$ since $\sum_{s=1}^\infty h(s)<\infty$.
The fact that $c\in\bbQ$ implies that $|A(N)|\leq |I(N)|<\hat CN$ for some $\hat C>0$
and observe that $\bbN-c\bbN=\{zb\}_{z\in\bbZ}$, 
where $b=\gcd(p,q)/q$ if $c=p/q$. Therefore, 
\begin{equation}\label{5.12}
\lim_{N\to\infty}\frac1N\sum_{x\in A(N)}\sum_{s=1}^{\infty}h(C_8(|x|+s))=0.
\end{equation}

It remains to show that  the second term on the right hand side of (\ref{5.10+})
converges to $\min(\al,\be)D_{i,j}$. 
 Observe that
(\ref{5.0}) is satisfied if and only if  $g_u(x_{i,j})=0$ for any $0<u<m$. 
On the one hand, 
suppose that $q_i$ and $q_j$ are not equivalent. Then
(\ref{5.0}) is not satisfied, for any $x_{i,j}\in\bbQ$, since $c\in\bbQ$. 
It follows that that for any $n,l$ there exists $0<u<m$ such that $g_u(n-cl)\neq0$. 
This means that $C(N)=\emptyset$,
and so $D_{i,j}(\al,\be)=0$ by (\ref{5.10+})-(\ref{5.12}), as
 claimed  in the statement of Proposition \ref{prop5.3}. 
On the other hand, suppose that  $q_i$ and $q_j$ are equivalent. Then (\ref{5.0})
is satisfied with some $x_{i,j}\in\bbQ$ and  $g_u(x_{i,j})=0$ for any $0<u<m$. 
In view of (\ref{5.10+})-(\ref{5.12}), establishing (\ref{5.1*})
reduces to computing the limit of $N^{-1}\sum_{(n,l)\in C(N)}b_{i,j}(n,l)$. 
We first claim that 
\begin{eqnarray}\label{5.12+}
&C(N)=B(N,x_{i,j}).
\end{eqnarray}
 Indeed, $g_{m-1}$
is linear and $x_{i,j}$ is its unique root, implying that  $n-cl=x_{i,j}$ for any $(n,l)\in C(N)$. 
The opposite inclusion is clear since $g_u(x_{i,j})=0$ for any $u=1,2,...,m-1$.  
It follows that the sets $\{b_{i,j}(n,l): (n,l)\in C(N)\}$ and 
$\{b_{i,j}(cl+x_{i,j},l): cl+x_{i,j}\in\bbN, 1\leq l\leq Nc_{j,i_{k-1}+1}\min(\al,\be)\}$
differ by (at most) finite number of elements which do not depend on $N$. Thus, 
\begin{equation}\label{5.12++}
\lim_{N\to\infty}\frac1N\sum_{(n,l)\in C(N)}b_{i,j}(n,l)=
\lim_{N\to\infty}\frac1N\sum_{l=1}^{Nc_{j,i_{k-1}+1}\min(\al,\be)}b_{i,j}(cl+x_{i,j},l)
\bbI_{\{cl+x_{i,j}\in\bbN\}}
\end{equation}
where we used that $|b_{i,j}(n,l)|$ is bounded in $n,l\in\bbN$, by Lemma \ref{lem5.1}. 
Equality (\ref{5.12++}) holds true, of course, only when the right hand side limit exists,
and the following arguments' purpose is proving its existence and  computing  it.
First, suppose that $i=j$. Then 
 $c=1$ and $x_{i,j}=0$ and the right hand side of   (\ref{5.12++}) becomes
\begin{equation}\label{5.15}
\lim_{N\to\infty}\frac1N\sum_{l=1}^{N\min(\al,\be)c_{j,i_{k-1}+1}}b_{j,j}(l,l).
\end{equation}
By   Lemma 4.3 from \cite{KV}  we have
$\lim_{l\to\infty}b_{j,j}(l,l)=\int F_j^2(z_1,...z_j)
d\mu(z_1)\cdot\cdot\cdot d\mu(z_j)$ 
and so  this limit exists and equals
$\min(\al,\be)D_{j,j}$, where $D_{j,j}$ is defined by (\ref{5.1+}). 

Suppose now that $i<j$ and recall  that (\ref{5.0}) is satisfied. Observe that
the asymptotic density
of the set of  $l$'s such that $cl+x_{i,j}\in\bbN$ equals $M_{i,j}/a(i,j)$,
where $M_{i,j}=M(a(i,j),b(i,j),x_{i,j})$ is defined by (\ref{4.01}).
Let  $l\in\bbN$ and consider the set 
\[
\Gam_l=\{q_s(cl+x_{i,j}):1\leq s\leq i\}\cup\{q_t(l):1\leq t\leq j\}
\]
of indexes  appearing in the summands of the right hand side of (\ref{5.12++}).
Similarly to the proof of Proposition \ref{prop5.2}, 
consider the polynomials $q_s(cy+x_{i,j})$ and $q_t(y)$, where $1\leq s\leq i$, 
$1\leq t\leq j$,  
and let $d$ be the number of distinct polynomials among them. Then
by (\ref{5.0-}),  there exists $L>0$ such that 
\begin{eqnarray*}
&\Gam(l)=\{m_1(l)<m_2(l)<...<m_d(l)\}\text{ for any } l>L
\end{eqnarray*}
and each $m_a$, $1\leq a\leq d$ is one of the above polynomials.  Observe that 
the polynomials $q_s(cy+x_{i,j})$ and $q_t(y)$ differ by a constant 
if and only if $(s,t)=(s_k,t_k)$ for some $k=1,2,...,r$, where  $(s_k,t_k)$ are defined
above (\ref{5.0+}) and we used that $c_{i,j}=c^{-1}$, (\ref{5.0}) and 
the change of variables $y\to c^{-1}(y-x_{i,j})$. On the other hand, 
if these polynomials do not differ by a constant then 
$\lim_{y\to\infty}|q_s(cy+x_{i,j})-q_t(y)|=\infty$. We conclude that for any 
$0<i<d$ either 
$\lim_{l\to\infty}m_{i+1}(l)-m_i(l)=\infty$ or $m_{i+1}(l)-m_i(l)$ is a constant. Thus, 
 Lemma 4.3 from \cite{KV} implies that
\begin{equation}
\lim_{l\to\infty,\,\,cl+x_{i,j}\in\bbN}b_{i,j}(cl+x_{i,j},l)=\int F_i(x)F_j(y)dm_{i,j}(x,y)
\end{equation}
where  $m_{i,j}$ is defined by (\ref{5.1}). We conclude that
 the right hand side of  (\ref{5.12++}) converges to
\[
D_{i,j}(\al,\be)=\frac{\min(\al,\be)c_{j,i_{k-1}+1}M_{i,j}}
{a(i,j)}\int F_i(x)F_j(y)dm_{i,j}(x,y)
\]
 and (\ref{5.1*}) follows by
(\ref{5.10+}), (\ref{5.11})  (\ref{5.12}) and (\ref{5.12++}).

We remark that in the terminology of Lemma 4.3 from \cite{KV} 
we used a partition into "rigid blocks" consisting of  pairs of the form  
$\{q_{s_k}(cl+x_{i,j}),q_{t_k}(l)\}$, $1\leq k\leq r$ and of singletons $\{q_s(cl+x_{i,j})\}$
and $\{q_t(l)\}$ where $t\in\{1,...,j\}\setminus\{t_1,...,t_r\}$ and 
$s\in\{1,...,i\}\setminus\{s_1,...,s_r\}$. 
\end{proof}

We conclude this section by showing that
$D^2=\lim_{N\to\infty}E\xi^2_N(1)$
exists, as claimed in Theorem \ref{thm2.2}. 
Again, relying on Section \ref{sec3}, it is sufficient to
prove this only in  case that $\hat\ell=\ell$. In this case
$r_i=i$, $i=1,...,\ell$, and by (\ref{2.19}),
\[
E\xi^2_N(1)=\sum_{A_1,A_2\in\mathcal A}E[\xi_N^{(A_1)}\xi_N^{(A_2)}]. 
\]
By Proposition \ref{prop5.3}, the  summands above converge to $0$
as $N\to\infty$ if $A_1\not=A_2$. 
On the other hand, when $A_1=A_2=A$,
Propositions \ref{prop5.2} and \ref{prop5.3} imply that
\begin{equation}\label{5.13}
\lim_{N\to\infty}E[\big(\xi_N^{(A)}\big)^2]=
\sum_{s:q_{r_s}\in A}c_{i_A,r_i}D_{i,i}
+2\sum_{i<j:q_{r_i},q_{r_j}\in A}c_{i_A,r_i}D_{i,j}
\end{equation}
and so $D^2$ exists and satisfies (\ref{2.20}). 
\qed

\section {Central limit theorem}\label{sec6}\setcounter{equation}{0}

In this section we complete the proof of Theorem \ref{thm2.2}. 
We assume without  loss of generality that $\hat\ell=\ell$, which is possible in view of 
of Section \ref{sec3}.  In this case
$r_i=i,\,i=0,1,...,\ell$ and
\begin{equation*}
\xi_{i,N}(t)=\frac1{\sqrt N}\sum_{n=1}^{[Ntc_{i,i_{k-1}+1}]}
F_i\big(X(q_1(n)),...,X(q_i(n))\big).
\end{equation*}
There exists $R\in\bbN$ such that the functions
 $q_i,q_i^{-1}$ and $q_i^{-1}\circ q_j$,
 $1\leq \,i,j\leq\ell$
 are well defined on $[R,\infty)$ and are strictly increasing and positive there.
Thus, by beginning the summation in the definition of $\xi_{i,N}$ from $R$ we can assume 
without  loss of generality that  these functions are well defined, positive and 
strictly increasing on $[0,\infty)$, which will simplify some of our arguments.

\subsection{Notations and approximations similar to \cite{HK} and \cite{KV}}\label{Sec6.1}

We introduce first the following notations from \cite{KV},
\begin{eqnarray}\label{6.1}
&F_{i,n,r}(x_1,...,x_{i-1},\om)=E(F_i(x_1,...,x_{i-1},X(n))|\cF_{n-r,n+r}),\\
&X_r(n)=E(X(n)|\cF_{n-r,n+r}),\nonumber \\
&Y_{i,q_i(n)}=F_i(X(q_1(n)),...,X(q_i(n)))\,\,\,\mbox{and}\,\,\,Y_{i,m}=0
\mbox{ if }m\notin\{q_i(l)\}_{l=1}^{\infty},\nonumber \\
&Y_{i,q_i(n),r}=F_{i,q_i(n),r}(X_r(q_1(n)),...,X_r(q_{i-1}(n)),\om)
\mbox{ and }\nonumber\\
&\,\,\,\,Y_{i,m,r}=0\,\mbox{ if
}\,m\notin\{q_i(l)\}_{l=1}^\infty\nonumber
\end{eqnarray}
where $1\leq i\leq \ell$ and $r$ is a nonnegative integer. 
We shall use the following estimate, as well.  There exists a constant $c_0>0$ such that  
\begin{equation}\label{6.2}
\|Y_{i,n,r}-Y_{i,n}\|_2\leq c_0\big(\beta_q(r)\big)^\del\,\,\,\mbox{for any}\,\,\,
r,n\in\bbN\,\, \mbox{and}\,\,\, 1\leq i\leq\ell.
\end{equation}
(\ref{6.2}) follows
 by Theorem 3.4 in \cite{KV}, relying on
 Assumption \ref{ass2.1} and  on (\ref{2.4})-(\ref{2.5}), and its proof goes in a similar 
way to the proof of Lemma 3.12 from the early preprint version arXiv:1012.2223v2
 of \cite{KV}. Notice that (\ref{6.2}) and (\ref{2.13}) imply that 
$Y_{i,n,r}=\lim_{r\to\infty}Y_{i,n,r}$ where the limit is taken in the $L^2(\Om,P)$
sense, and it follows that 
\begin{equation}\label{6.2+}
Y_{i,n}=Y_{i,n,2^u}+\sum_{m=u+1}^\infty Y_{i,n,2^m}- Y_{i,n,2^{m-1}}
\end{equation}
for any $1\leq i\leq\ell$, $n\in\bbN$ and $u\in\bbN$, where we used (\ref{2.13}).
We shall need also the following estimate. For any $T>0$ there exists a constant 
$C_T>0$ such that 
\begin{equation}\label{6.2++}
\sum_{r\geq 0}\sup_{N\geq1}\max_{1\leq i\leq\ell}\big\|\sup_{0\leq t\leq T}
|\frac1{\sqrt N}\sum_{1\leq n\leq Nt} Y_{i,q_i(n),2^r}-Y_{i,q_i(n),2^{r-1}}|\big\|_2
\leq C_T,
\end{equation}
where $Y_{i,q_i(n),2^{-1}}:=0$. The proof of   (\ref{6.2++})
goes exactly as the proof of Proposition 5.9 from \cite{KV}.
Next for any $u\in\bbN$,  $k=1,...,v$ and $i_{k-1}<i\leq i_{k}$ set
\begin{eqnarray*}
\xi_{i,N}^{(u)}(t)=\frac1{\sqrt N}\sum_{n=1}^{[Ntc_{i,i_{k-1}+1}]}
Y_{i,q_i(n),2^u}.
\end{eqnarray*}
Then  by (\ref{6.2+}) and
(\ref{6.2++}), for any $T>0$ and  $i=1,...,\ell$,
\begin{eqnarray}\label{6.2+++}
&\lim_{u\to\infty}\sup_{N\in\bbN}\big\|\sup_{0\leq t\leq T}|\xi_{i,N}(t)-
\xi_{i,N}^{(u)}(t)|\big\|_2=0.
\end{eqnarray}

The last estimate we need goes as follows. There exists a constant $C>0$ such that for any strictly
increasing sequence of natural numbers $\{n_s\}_{s=1}^\infty$, $1\leq i\leq\ell$ and
$u,m,l\in\bbN$, 
\begin{equation}\label{Lem5.2}
\big\|\sum_{s=l}^{l+m-1}Y_{i,q_i(n_s),2^{u-1}}\big\|_2, 
\big\|\sum_{s=l}^{l+m-1}Y_{i,q_i(n_s)}\big\|_2\leq C\sqrt m.
\end{equation}
This result was proved  in Lemma 5.2 from  \cite{HK}  (with $b=2$) in case that 
$q_i(n)=in$, $i=1,2,...,\ell$.  The proof from \cite{HK} relies
on the mixing rates (\ref{2.12})-(\ref{2.13}) and on  Corollary 3.6  from \cite{KV}
together with the inequality
$q_{i+1}(n)-q_i(n)\geq n$ for any $i$ and sufficiently large $n$. Assuming that $\ell=\hat\ell$,
 there exists $C_1>0$ such that 
$q_{i+1}(n)-q_{i}(n)\geq C_1n$ for any $i$ and  sufficiently large $n$,
and so the proof of  (\ref{Lem5.2}) proceeds in our setup in the
same way.

\subsection{Proof of Theorem \ref{2.2}}\label{subsec6.1}

As pointed out in Section \ref{sec1}, we are going to adapt  the martingale approximation
approach from \cite{KV} to our  situation. 
We begin with showing that the process $\big(\xi_{i,N}(\cdot)\big)_{i=1}^\ell$ 
is tight when considered as $D\big([0,\infty);\bbR^{\ell}\big)$ (the $\ell$- dimensional Skorokhod space)
valued random variable. 
The arguments from and below either
 Proposition \ref{prop6.1} or Proposition 5.8 from \cite{KV} 
show that when $i$ is fixed Theorem \ref{Thm5.1} is applicable with appropriate
subsequences of $\xi_{i,N}^{(u)}(\cdot)$, and that by letting $u\to\infty$,
 each  one dimensional  component of  $\big(\xi_{i,N}(\cdot)\big)_{i=1}^\ell$ weakly converges
as $N\to\infty$. 
In particular, each one of them forms a tight sequence of  
$D\big([0,\infty);\bbR\big)$ valued random variables.
The (non random) lattice structure of the discontinuity points of the process
$\big(\xi_{i,N}(\cdot)\big)_{i=1}^\ell$  together with  Theorem 3.21 from
Chapter VI in \cite{JS} 
imply that tightness of the $\ell$ dimensional  process
$\big(\xi_{i,N}(\cdot)\big)_{i=1}^\ell$ 
follows from tightness of its one dimensional components.
We conclude that  weak converge of $\big(\xi_{i,N}(\cdot)\big)_{i=1}^\ell$ 
 follows from weak  convergence of its finite dimensional distributions.

Next,  let $1\leq k\leq v$,  $i_{k-1}<i\leq i_k$, $n\in\bbN$ and $r>0$ and  set
\begin{eqnarray}\label{Defs}
&A_i=q_{i_{k-1}+1}^{-1}(q_i(\bbN)),\,\,
A^{(k)}=\bigcup_{i=i_{k-1}+1}^{i_k}A_i,\\
& b_n^{(k)}=q_{i_{k-1}+1}(z_n^{(k)}),\,\,\,
\nu_i(r)=|A_i\cap(0,r]|\,\,\text{and }\,
\nu^{(k)}(r)=|A^{(k)}\cap(0,r]|\nonumber
\end{eqnarray}
where $|\Gam|$ denotes the cardinality of  a finite set $\Gam$. 
Observe that $b_n^{(k)}\in\bbN$ and that $\lim_{r\to\infty}\nu_i(r)/r=
c_{i,i_{k-1}+1}>0$,  
since $\lim_{y\to\infty} q_{i_{k-1}+1}^{-1}(q_i(y))/y=c_{i_{k-1}+1,i}$. 
By   Corollary \ref{cor4.1}, the set  $A^{(k)}$ has the form $A^{(k)}=\{z_1^{(k)}<z^{(k)}_2<...\}$
and there exists $c^{(k)}\geq 1$ such that
\begin{eqnarray}\label{6.4}
&\lim_{r\to\infty}\frac{\nu^{(k)}}r=c^{(k)}\,\,\text{ and }\,\,
\lim_{m\to\infty}\frac {z_m^{(k)}}m=\frac1{c^{(k)}}.
\end{eqnarray}
We note also that $\nu_i(r)=\max\{m: q_{i_{k-1}+1}^{-1}(q_i(m))\leq r\}$ and 
$\nu^{(k)}(r)=\max\{m: z_m^{(k)}\leq r\}$, 
for any sufficiently large $r$.

The next step of the proof is to approximate
the process $\big(\xi_{i,N}(t)\big)_{i= i_{k-1}+1}^{i_k}$ by 
the process $\hat\Psi_{k,N}(tc^{(k)})$ defined below, for which Proposition 
\ref{Cor5.7} is applicable. 
First, (\ref{Lem5.2}) yields that for any
 $n,u\in\bbN$, $t>0$, $1\leq k\leq v$ and $i_{k-1}< i\leq i_k$, 
\begin{eqnarray}\label{6.6}
&\,\,\,\,\,\,\,\|\xi_{i,N}(t)-\frac1{\sqrt N}\sum_{n=1}^{\nu_i(Nt)}Y_{i,q_i(n)}\|_2\,,\,\,
\|\xi_{i,N}^{(u)}(t)-\frac1{\sqrt N}\sum_{n=1}^{\nu_i(Nt)}Y_{i,q_i(n),2^u}\|_2
\leq\\
&C\sqrt{\frac{|\nu_i(Nt)-[Ntc_{i,i_{k-1}+1}]|}N}\leq CN^{-\frac12}+Ct^
{\frac12}\sqrt{|\frac{\nu_i(Nt)}{Nt}-c_{i,i_{k-1}+1}|}\nonumber
\end{eqnarray}
where $C>0$ is independent of $N,u$ and $t$.
Second,  observe that
\begin{equation}\label{6.5}
\sum_{n=1}^{\nu_i(r)}Y_{i,q_i(n)}=
\sum_{n=1}^{\nu^{(k)}(r)}Y_{i,b^{(k)}_n}
\,\,\,\,\mbox{and}\,\,\,\,
\sum_{n=1}^{\nu_i(r)}Y_{i,q_i(n),2^u}=\sum_{n=1}^{\nu^{(k)}(r)}Y_{i,b^{(k)}_n,2^u}
\end{equation}
for any $n\in\bbN$, $u\geq 0$, $r>0,$ $1\leq k\leq v$ and $i_{k-1}< i\leq i_k$. Set
\begin{eqnarray*}
&\Psi_{N,k}(t)=\frac1{\sqrt N}\big(\sum_{n=1}^{\nu^{(k)}(Nt)}Y_{i,b^{(k)}_n}
\big)_{i=i_{k-1}+1}^{i_{k}}\text{ and }\hskip2cm\\
&\Psi_{N,k}^{(u)}(t)=\frac1{\sqrt N}\big(\sum_{n=1}^{\nu^{(k)}(Nt)}Y_{i,b^{(k)}_n
,2^{u}}\big)_{i=i_{k-1}+1}^{i_k}.
\end{eqnarray*}
Similarly to (\ref{6.6}), for any $1\leq k\leq v$, $t>0$ and $N\in\bbN$, 
\begin{eqnarray}\label{6.7}
&\|\Psi_{N,k}(t)-\hat\Psi_{N,k}(c^{(k)}t)\|_2\,,\,\,
\|\Psi_{N,k}^{(u)}(t)-\hat\Psi_{N,k}^{(u)}(c^{(k)}t)\|_2\leq\\
&CN^{-\frac12}+
Ct^{\frac12}\sqrt{\big|\frac{\nu^{(k)}(Nt)}{Nt}-c^{(k)}\big|}\nonumber
\end{eqnarray}
where
\begin{eqnarray*}
\hat\Psi_{N,k}(t)=\frac1{\sqrt N}\big(\sum_{n=1}^{[Nt]}Y_{i,b^{(k)}_n}\big)
_{i=i_{k-1}+1}^{i_k}\text{ and }
\hat\Psi_{N,k}^{(u)}(t)=\frac1{\sqrt N}\big(\sum_{n=1}^{[Nt]}
Y_{i,b^{(k)}_n,2^u}\big)_{i=i_{k-1}+1}^{i_k}\nonumber.
\end{eqnarray*}
Now by (\ref{6.6})-(\ref{6.7}),  for any $t>0$,
\begin{eqnarray}\label{6.7+}
&\lim_{N\to\infty}\|(\xi_i(t))_{i=1}^\ell-(\hat\Psi_k(c^{(k)}t))_{k=1}^v\|_2=0.
\end{eqnarray}
Next, by (\ref{6.2+}) and
(\ref{6.2++}), for any $T>0$ and $k=1,...,v$,
\begin{eqnarray}\label{6.9'}
&\lim_{u\to\infty}\sup_{N\in\bbN}\big\|\sup_{0\leq t\leq T}|\hat\Psi_{k,N}(t)-
\hat\Psi_{k,N}^{(u)}(t)|\big\|_2=0
\end{eqnarray}
where  we used (\ref{6.5}) with $r=z^{(k)}_{[Nt]}$ and that there exists a constant $L$
such that $\nu_i(z^{(k)}_r)\leq Lr$ for any $r>0$ and $i_{k-1}<i\leq i_k$.

We conclude that in order to prove that the process
$\big(\xi_{i,N}(\cdot)\big)_{i=1}^\ell$ weakly converges as $N\to\infty$, it is sufficient to 
show that  the process $\hat\Psi_N(t)=\big(\hat\Psi_{N,k}(t)\big)_{k=1}^v$  weakly
converges, and then to  plug
in $tc^{(k)}$ in place of $t$ in  the coordinates at places $i_{k-1}+1,...,i_k$, namely 
to replace $\hat\Psi_{N,k}(t)$ with $\hat\Psi_{N,k}(tc^{(k)})$. 
In order to obtain the desired weak convergence of
$\hat\Psi_N$
it is important to understand the asymptotic behavior of the approximation
processes $\hat\Psi_N^{(u)}=\big(\hat\Psi_{N,k}^{(u)}\big)_{k=1}^v$  by describing 
their sets of  limit points.

\begin{proposition}\label{prop6.1}
For each fixed $u\in\bbN$  when $N\to\infty$
the processes  $\hat\Psi_N^{(u)}$ 
form a tight family of processes in the Skorokhod space
$D\big([0,T];\bbR^\ell\big)$.
All limit points have the form $\zeta_1\times\zeta _2\times\cdots\times\zeta_v$, where 
$\zeta_k=\zeta_k^{(u)}, 1\leq k\leq v$ is a centered Gaussian process  with
independent increments. The second moments of each $\zeta_k$
are uniformly integrable so that
the  covariances of the limiting Gaussian process $\zeta_k$ can be identified as the limits
of the corresponding covariances of the approximating processes
along a subsequence.
\end{proposition}
\begin{proof}
Fix some $1\leq k\leq v$ and set $z_n=z_n^{(k)}$ and
$b_n=b_n^{(k)}=q_{i_{k-1}+1}(z_n)$. Then $(b_n)_{n=1}^\infty$ is a
monotone increasing sequence of natural numbers. We first prove that 
when $N\to\infty$
the processes  $\hat\Psi_{k,N}^{(u)}$ 
form a tight family of processes in the Skorokhod space
$D\big([0,T];\bbR^{i_k-i_{k-1}}\big)$, and then identify the limit points.
For any $N\in\bbN$ consider the filtration
$\{\cG^{(u)}_{k,N,n},\,n\geq1\}$ where
$\cG^{(u)}_{k,N,n}=\cG^{(u)}_{k,n}=\cF_{-\infty,b_n+2^u}$ and let
the adapted random vectors  $\{Z^{(u,k)}_{N,n},\,n\geq1\}$ be defined
 by $Z_{N,n}^{(u,k)}=(Z^{(u)}_{N,i,n})_{i=i_{k-1}+1}^{i_k}$, where  
$Z^{(u)}_{N,i,n}=Z^{(u)}_{i,n}=Y_{i,b_n,2^u}$. Then
\begin{eqnarray}\label{H-Psi-rep}
&\hat\Psi_{k,N}^{(u)}(t)=\frac1{\sqrt
N}\sum_{n=1}^{[Nt]}Z^{(u,k)}_{N,n}.
\end{eqnarray}
Next, we show  that for any $i_{k-1}<i\leq i_k$ the 
one dimensional process $\{Z_{N,i,n}^{(u)},\,n\geq1\}$
satisfies conditions B1-B3 from Theorem \ref{Thm5.1}
with the filtration
$\{\cG^{(u)}_{k,N,n},\,n\geq1\}$, which clearly
implies that any linear combination $\{\langle\la, Z_{N,n}^{(u,k)}\rangle,n\geq1\}$
satisfies conditions B1-B3 with this filtration. Indeed, fix some $i_{k-1}<i\leq i_k$.
Condition B1  is just $\cG^{(u)}_{k,N,n}$-measurability  of $Z_{N,i,n}^{(u)}$ and 
Condition B2 is verified exactly as in
Proposition 5.8 from \cite{KV}. 

Before verifying Condition B3, we need the following simple observation. 
We claim that there exist constants $M_0,A_0>0$ such that
\begin{eqnarray}\label{(2.5)}
b_m-q_{i-1}(q_i^{-1}(b_m))\geq A_0m \,\text{ for any } m\geq M_0
\end{eqnarray}
where we set $q_0=0$ in case that $i=1$. 
Indeed, 
since  $q_i(y)-q_{i-1}(y)$  is a polynomial converging to $\infty$
as $y\to\infty$, there exist $M_1>0$ and $C>0$ such that 
$q_i(y)-q_{i-1}(y)\geq Cy$ for any $y\geq M_1$.   
By  (\ref{6.4}) and the definition of $b_l$ we have
$\lim_{l\to\infty}q_i^{-1}(b_l)/l=c_{i,i_{k-1}+1}/c^{(k)}$ and since we assumed that
$q_i^{-1}\circ q_{i_{k-1}+1}$ is positive on $[0,\infty)$, there exists $C'>0$
such that $q_i^{-1}(b_l)\geq C'l$ for any $l\in\bbN$. Set $\hat m=q_i^{-1}(b_m)$. 
Then  $\hat m\geq C'm$,   
and hence if $m\geq M_0=M_1/C'$ then $\hat m\geq M_1$ and so
\[ 
b_m-q_{i-1}(q_i^{-1}(b_m))=q_i(\hat m)-q_{i-1}(\hat m)\geq C\hat m\geq CC'm
\]
and (\ref{(2.5)}) follows with $A_0=CC'$.

Now we show that Condition B3 is satisfied. We have to control
 $\|E(Y_{i,b_m,2^u}|\cF_{-\infty,b_n+2^u})\|_2$
for $m\geq n$. Notice that it vanishes unless $b_m=q_i(\hat m)$
for some $\hat m\in\bbN$, and so we consider only  this case. On the one hand, 
 if $q_{i-1}(\hat m)\leq b_n$ and $b_m=q_i(\hat m)\geq b_n+2^{u+1}$
 then by (\ref{2.17}) together with Corollary 3.6(ii) from \cite{KV} 
  we obtain that
\begin{eqnarray}\label{(1)}
&\|E(Y_{i,b_m,2^u}|\cF_{-\infty,b_n+2^u})\|_2=
\|E(Y_{i,q_i(\hat m),2^u}|\cF_{-\infty,b_n+2^u})\|_2\leq\\
& c_1\vp_{q,p}(q_i(\hat m)-b_n-2^{u+1})=
c_1\vp_{q,p}(b_m-b_n-2^{u+1}) \leq c_1\vp_{q,p}(m-n-2^{u+1})\nonumber
\end{eqnarray}
where $c_1>0$ depends on the parameters
 $d,p,\ka,\iota,m,q,K$ from Assumption \ref{ass2.1}
and  (\ref{2.4})-(\ref{2.5}).
The last inequality holds true if $m>n+2^{u+1}$ and we used that 
$b_m-b_n\geq m-n$ which is satisfied since
$(b_n)_{n=1}^\infty$ is strictly increasing and takes  natural values.
On the other hand, if $q_{i-1}(\hat m)\geq b_n$ then by the contraction
property of conditional expectations similarly to  (\ref{(1)}) we have
\begin{eqnarray}\label{(2)}
&\|E(Y_{i,b_m,2^u}|\cF_{-\infty,b_n+2^u})\|_2=
\|E(Y_{i,q_i(\hat m),2^u}|\cF_{-\infty,b_n+2^u})\|_2\leq\\
&\| E(Y_{i,q_i(\hat m),2^u}|\cF_{-\infty,q_{i-1}(\hat m)+2^u})\|_2\leq
c_1\vp_{q,p}(q_i(\hat m)-q_{i-1}(\hat m)-2^{u+1})\nonumber
\end{eqnarray}
where the second inequality holds if $q_i(\hat m)-q_{i-1}(\hat m)>2^{u+1}$. 
Next, 
set $K_0=\max(M_0,2^{u+2}/A_0)$ where $M_0$ and $A_0$ satisfy
(\ref{(2.5)}). 
Then $A_0m-2^{u+1}\geq A_0m/2$ for any 
$m\geq K_0$, and  we conclude by (\ref{(2.5)}) and (\ref{(2)}) that for any $m\geq K_0$,
\begin{equation}\label{(3)}
\|E(Y_{i,b_m,2^u}|\cF_{-\infty,b_n+2^u})\|_2
\leq c_1\vp_{q,p}(A_0m-2^{u+1})\leq c_1\vp_{q,p}(\frac{mA_0}2)
\end{equation}
assuming that $q_{i-1}(\hat m)\geq b_n$, 
where $\vp_{q,p}(s)=\vp_{q,p}([s])$ for any $s\geq 0$. 

Finally, suppose that neither (\ref{(1)}) nor
(\ref{(3)}) can be applied, namely that $m$ does not satisfy
$q_{i-1}(\hat m)\leq b_n\,,\,q_i(\hat m)\geq b_n+2^{u+1}\,\,and\,\,m>n+2^{u+1}$ 
or $q_{i-1}(\hat m)\geq b_n\,\,and\,\,m\geq K_0$. Then either 
$m\leq \max(n+2^{u+1},K_0)$ or $b_m=q_i(\hat m)<b_n+2^{u+1}$. The 
last inequality  implies that $m-n<2^{u+1}$,  since $n-m\leq b_m-b_n$. 
Thus, there exist at most $2^{u+1}+K_0$ naturals $m\geq n$ such that  
neither (\ref{(1)}) nor (\ref{(3)}) can be applied. 
 Approximating
$\|E(Y_{i,b_m,2^u}|\cF_{-\infty,b_n+2^u})\|_2$ by $\|Y_{i,b_m,2^u}\|_2$
 for such $m$'s (using validity of Condition B2), 
we see that there exist constants $c_2,C_0>0$ such that 
\[
\sum_{m=n}^{\infty}\|E(Y_{i,b_m,2^u}|\cF_{-\infty,b_n+2^u})\|_2
\leq c_2\big(2^{u+1}+
K_0+\sum_{s=1}^{\infty}\vp_{q,p}(C_0s)\big)<\infty,
\]
where the right hand side is finite 
 in view of (\ref{2.12}),  and  Condition B3 is satisfied. 

Next, we prove that the family of processes $\hat\Psi_{k,N}^{(u)}$ is tight when $N\to\infty$, 
and specify the corresponding limit points. We 
start by showing that along suitable subsequences any linear combination of
its one dimensional components satisfy (\ref{B4'}), and so also Condition B4 from 
Theorem \ref{Thm5.1} (see Appendix).
Indeed, set 
$W_{N,n}^{(u,k)}=(W_{N,i,n}^{(u)})_{i=i_{k-1}+1}^{i_k}$, where
\[
 W^{(u)}_{N,i,n}=Z^{(u)}_{N,i,n}
+\sum_{s\geq n+1}E(Z^{(u)}_{N,i,s}|\cG^{(u)}_{k,N,n})-\sum_{s\geq n}
E(Z^{(u)}_{N,i,s}|\cG^{(u)}_{k,N,n-1})
\]
and $Z^{(u)}_{N,i,n}$ is defined before (\ref{H-Psi-rep}).
Let $\{\hat\Psi_{N_j,k}^{(u)}(\cdot),j\geq1\}$ be any subsequence. The
uniform integrability (Condition B2) together with validity of Condition B3
 imply that $\|W_{N,i,n}^{(u)}\|_2\leq C_1$
for some $C_1>0$ which is independent of $i,n$ and $N$.
Therefore, a diagonal argument  shows that
we can pick a subsequence $\{N_{j_z},z\geq1\}\subset\{N_j,j\geq1\}$ such that
for any $i_{k-1}<i,i'\leq i_k$ the limit
\[
A_{i,i'}^{(u)}(t)=\lim_{z\to\infty}\frac1{N_{j_z}}\sum_{n=1}^{[tN_{j_z}]}
E(W_{N_{j_z},i,n}^{(u)})E(W_{N_{j_z},i',n}^{(u)})
\]
 exists on a dense subset of $[0,T]$. Observe now that
 \[
|\frac1N\sum_{n=1}^{tN}E(W_{N,i,n}^{(u)})E(W_{N,i',n}^{(u)})-
\frac1N\sum_{n=1}^{sN}E(W_{N,i,n}^{(u)})E(W_{N,i',n}^{(u)})|\leq
C_1^2|t-s|
\]
for any $N$, $i_{k-1}<i,i'\leq i_k$  and $s,t\geq 0$, 
and so  this limit exists for any $t\in[0,T]$ and $i_{k-1}<i,i'\leq i_k$. 
Next, let $\la=(\la_i)_{i=i_{k-1}+1}^{i_k}\in\bbR^{i_k-i_{k-1}}$
and consider the linear combinations
\begin{eqnarray*}
\langle \la, Z_{N,n}^{(u,k)}\rangle
=\sum_{i_{k-1}<i\leq i_k}\la_iZ_{N,i,n}^{(u)}\,\mbox{ and }\,
\langle\la,W_{N,n}^{(u,k)}\rangle=\sum_{i_{k-1}<i\leq i_k}\la_iW_{N,i,n}^{(u)}.
\end{eqnarray*}
It follows that  for any $t\in[0,T]$,
 \begin{equation}\label{Cov}
\lim_{z\to\infty}\frac1{N_{j_z}}\sum_{n=1}^{tN_{j_z}}
E\big(\langle\la,W_{N,n}^{(u,k)}\rangle\big)^2=
\sum_{i_{k-1}<i,i'\leq i_k}\la_i\la_{i'}A_{i,i'}^{(u)}(t)
 \end{equation}
 which implies that (\ref{B4'}) is satisfied with
$U_{N,n}=\langle \la, Z_{N,n}^{(u,k)}\rangle$
 along  the subsequence $\{N_{j_z}, z\geq 1\}$ (which is independent of $\la$). 
 Therefore, applying  Proposition \ref{Cor5.7} with the subsequence
 $\{\hat\Psi_{N_{j_z},k}^{(u)}(\cdot),z\geq1\}$ we
 deduce that it converges to  a Gaussian process with independent
 increments and covariance matrix
$(A^{(u)}_{i,i'}(\cdot))_{i_{k-1}<i,i'\leq i_k}$.

Next, let $\{m_N\}_{N=1}^\infty\subset\bbN$ be a sequence satisfying 
$\lim_{N\to\infty}\frac {m_N}N=0$ and set  
\[
\tilde\Psi_{k,N}^{(u)}(t)=\frac1{\sqrt N}\sum_{m_N+1\leq n\leq Nt+m_N}
\big(Z_{N,i,n}^{(u)}\big)_{i=i_{k-1}+1}^{i_k}\,\,\text{ and } 
\tilde\Psi_N^{(u)}=\big(\tilde\Psi_{k,N}^{(u)}\big)_{k=1}^v.
\]
Then by (\ref{H-Psi-rep}) and (\ref{Lem5.2}),
\begin{equation}\label{Est}
\lim_{N\to\infty}
\big\|\sup_{0\leq t\leq T}|\hat\Psi_{k,N}^{(u)}(t)-\tilde\Psi_{k,N}^{(u)}(t)|\big\|_2=0
\text{ and } \lim_{N\to\infty}
\big\|\sup_{0\leq t\leq T}|\hat\Psi_N^{(u)}(t)-\tilde\Psi_N^{(u)}(t)|\big\|_2=0
\end{equation}
where the first equality holds true for any $1\leq k\leq v$. 
Let $\{\tilde\Psi_{k,N_j}^{(u)}(\cdot),j\geq1\}$ be a subsequence  of 
$\{\tilde\Psi_{k,N}^{(u)}(\cdot),N\geq1\}$. 
Then 
by (\ref{Est}), this subsequence weakly converges if and only if the 
subsequence $\{\hat\Psi_{k,N_j}^{(u)}(\cdot),j\geq1\}$ of 
$\{\hat\Psi_{k,N}^{(u)}(\cdot),N\geq1\}$  weakly converges, and in
this case they converge to the same limit. Similarly, a
subsequence 
 $\{\tilde\Psi_{N_j}^{(u)}(\cdot),j\geq1\}$ of $\{\tilde\Psi_N^{(u)}(\cdot),N\geq1\}$
weakly converges if and only if the subsequence
$\{\hat\Psi_{N_j}^{(u)}(\cdot),N\geq1\}$  weakly converges,
 and in this case they converge to the same limit. 

Now we show that the processes 
$\hat\Psi_N^{(u)}=\big(\hat\Psi_{N,k}^{(u)}\big)_{k=1}^v$, $N\in\bbN$ form a tight family of 
processes.
 Set 
\[
m_N=m_{T,N}=\min\{n: q_{i_{k-1}+1}(z_n^{(k)})\geq q_{i_{k-2}+1}(z_{NT}^{(k-1)})
,\,\,\,\forall  k=2,3,...,v\}.
\]
Then, $\lim_{N\to\infty}\frac{m_N}N=t_0=0$ 
taking into account (\ref{6.4}) and that $\deg{q_{i_{k-1}+1}}>\deg{q_{i_{k-2}+1}}$.
As explained above, it is sufficient to prove that the 
processes $\tilde\Psi_N^{(u)}$, $N\in\bbN$ form a tight family. Let 
$\{\tilde\Psi_{N_j}^{(u)}, j\geq1\}$ be a subsequence. Then there exists a subsequence 
$\{N_{j_z},z\geq1\}\subset\{N_j,j\geq1\}$ such that (\ref{Cov}) is satisfied 
for any $1\leq k\leq v$ and  $i_{k-1}<i,i'\leq i_k$. Thus, for any 
$1\leq k\leq v$ the processes
$\{\tilde\Psi_{k,N_{j_z}}^{(u)}, z\geq 1\}$ weakly  converge to a centered
 Gaussian limit $\zeta_k^{(u)}$
with the covariances matrix $(A_{i,i'}^{(u)}(\cdot))_{ i_{k-1}<i,i'\leq i_k}$.
Similarly  to (\ref{Est}) and below it, we can omit the last $m_N$
summands in the definition of $\tilde\Psi_{k,N}^{(u)}$ since $\lim_{N\to\infty}\frac{m_N}N=0$. 
A repetitive use of  Theorem \ref{Thm5.6} shows that the subsequence 
$\{\tilde\Psi_{N_{j_z}}^{(u)}, z\geq 1\}$ weakly converges to 
$\zeta_1^{(u)}\times\zeta_2^{(u)}\times\cdots\times\zeta_k^{(u)}$ and the
proof of Proposition \ref{prop6.1} is complete.  
\end{proof}

Now we  deduce the desired weak converges of $\hat\Psi_N=\big(\hat\Psi_{k,N}\big)_{v=1}^k$
 by  letting  $u\to\infty$. We will use the notations $Z_{N,i,n}^{(u)}$,
$W_{N,i,n}^{(u)}$ and $\cG_{k,N,n}^{(u)}$ from the proof of Proposition \ref{prop6.1}.
In view of Proposition \ref{prop6.1} and its proof
 it suffices to show that for any $1\leq k\leq v$, $i_{k-1}< i,i'\leq i_k$ and $T>0$,
\begin{equation}\label{6.10}
\lim_{u\to\infty}\limsup_{N\to\infty}\sup_{0\leq t\leq T}|
\frac1N\sum_{n=1}^{[Nt]}E(W_{N,i,n}^{(u)})E(W_{N,i',n}^{(u)})-
\frac t{c^{(k)}} D_{i,i'}|=0
\end{equation}
where $D_{i,i'}$'s were introduced in Theorem \ref{thm2.2}.
This together with (\ref{6.9'}) imply that $\hat\Psi_N(\cdot)$ converges as $N\to\infty$
 towards a centered Gaussian process with independent increments and covariance matrix 
$(A_{i,i'})_{1\leq i,i'\leq\ell}$ satisfying $A_{i,i'}(t)=tD_{i,i'}/c^{(k)}$ if $i_{k-1}<i,i'\leq i_k$ and otherwise 
$A_{i,i'}(t)=0$.

Establishing (\ref{6.10}),  let $i_{k-1}<i,i'\leq i_k$, $N,u\in\bbN$ and 
$T,t>0$.  For any $i_{k-1}<j\leq i_k$
the process $\{W^{(u)}_{N,j,n}, n\geq1\}$ is a martingale difference with respect to the filtration
$\{\cG_{k,N,n}^{(u)}, n\geq1\}$ and therefore,
\begin{equation}\label{6.10'}
\sum_{n=1}^{[Nt]}E[(W^{(u)}_{N,i,n})(W^{(u)}_{N,i',n})]=
E\big[(\sum_{n=1}^{[Nt]}W^{(u)}_{N,i,n})
(\sum_{n=1}^{[Nt]}W^{(u)}_{N,i',n})\big].
\end{equation}
Condition B3 verified  in Proposition \ref{prop6.1} implies that
$\|\sum_{n=1}^{[Nt]}W^{(u)}_{N,j,n}-\sum_{n=1}^{[Nt]}Z^{(u)}_{N,j,n}\|_2$
is bounded in $N$, $t\in[0,T]$ and $i_{k-1}<j\leq i_k$. 
This together with (\ref{6.10'}) and (\ref{Lem5.2}) shows that
(\ref{6.10}) follows from
\begin{equation}\label{6.9''}
\lim_{u\to\infty}\limsup_{N\to\infty}\sup_{0\leq t\leq T}
|\frac1NE[\big(\sum_{n=1}^{[Nt]}Z^{(u)}_{N,i,n}\big)
(\sum_{n=1}^{[Nt]}Z^{(u)}_{N,i',n}\big)]
-\frac t{c^{(k)}}D_{i,i'}|=0.
\end{equation}
Proving (\ref{6.9''}),   first by 
(\ref{6.7})  and (\ref{Lem5.2}) applied with $j=i,i'$,
the normalized sums $N^{-1/2}\sum_{n=1}^{[Nt]}Z^{(u)}_{N,j,n}$ can be replaced with
 $N^{-1/2}\sum_{n=1}^{\nu^{(k)}(t/c^{(k)})}Y_{j,b_n^{(k)},2^u}$, 
since the approximations in (\ref{6.7}) are uniform in $u$.
Second, by (\ref{6.5}) and then by (\ref{6.6}) the latter  sums can be replaced with 
$\xi_{j,N}^{(u)}(t/c^{(k)})$, where we used again  (\ref{Lem5.2}) and that
the approximations in (\ref{6.6}) are uniform in $u$. 
Third, by (\ref{6.2+++}) the sums $\xi_{j,N}^{(u)}(t/c^{(k)})$  can be replaced with
$\xi_{j,N}(t/c^{(k)})$
since  the approximation there is uniform in $N$ and $0\leq t\leq T$.
Now  (\ref{6.9''}) follows from
Propositions \ref{prop5.2} and \ref{prop5.3}.

Plugging in $c^{(k)}t$ in place of $t$ in the coordinates at places 
$i_{k-1}+1,...,i_k$ shows that 
 $\big(\hat\Psi_{N,k}^{(u)}(c^{(k)}t)\big)_{k=1}^v$ converges in distribution as
$N\to\infty$ to a Gaussian process $\zeta=\zeta_1\times\zeta_2\times\cdots\times\zeta_v$, 
$\zeta_k(t)=\big(\eta_i(t)\big)_{i=i_{k-1}+1}^{i_k}$, 
with stationary independent increments and covariance matrix satisfying
 $A_{i,j}(t)=tD_{i,j},\,1\leq i,j\leq \ell$, where  $D_{i,j}$'s are given by
Propositions \ref{prop5.2} and \ref{prop5.3}. The convergence of 
$(\xi_{i,N})_{i=1}^\ell$ follows now from its tightness and from (\ref{6.6})-(\ref{6.7}). 
\\

Finally, by (\ref{2.19})
the process $\xi_{N}(t)$ weakly converges  to $\eta(t)$ defined by (\ref{2.19++}). 
It remains to show that $\eta$ is a Gaussian process. Let
$1\leq k\leq v$ and $i\leq i_k<j$. Then  $\eta_i$ and $\eta_j$
are independent. Therefore,  it suffices to prove that
 $\sum_{i=i_{k-1}+1}^{i_k}\eta_i(c_{i_{k-1}+1,i}t)$ is a Gaussian process
 for each $k$. Indeed, set $d_i=c_{i_{k-1}+1,i}$. Then 
 $d_{i_{k-1}+1}\leq d_{i_{k-1}+2}\leq...\leq d_{i_k}$ and 
 observe that
\begin{eqnarray*}
&\sum_{i=i_{k-1}+1}^{i_k}\eta_i(td_i)=\sum_{j=i_{k-1}+1}^{i_k}
\sum_{i=i_{k-1}+1}^{i_k}\la_{i,j}(\eta_j(td_i)-\eta_j(td_{i-1}))=\\
&\sum_{i=i_{k-1}+1}^{i_k}
\sum_{j=i_{k-1}+1}^{i_k}\la_{i,j}(\eta_j(td_i)-\eta_j(td_{i-1}))
\end{eqnarray*}
where $\la_{i,j}=1$ if $i\leq j$ and $\la_{i,j}=0$ if $i>j$, $d_{i_{k-1}}=0$
and $\eta_j(0)=0$.
The increments of $\big(\eta_i(t)\big)_{i=i_{k-1}+1}^{i_k}$  are independent. Thus,
the vectors
$\Gam_i(t)=\{\la_{i,j}(\eta_j(td_i)-\eta_j(td_{i-1})):i_{k-1}<j\leq i_k\}\,,\,\,i_{k-1}<i
\leq i_{k}$ are independent and Gaussian which makes $\eta(\cdot)$ a
Gaussian process. The increments of $\eta(\cdot)$
are not necessary independent as shown  in \cite{KV}.
The counter example given there is in the case of two linear polynomials $q_1$ and $q_2$. 
 In Section \ref{sec7} we will give another counter example with
 nonlinear $q_i$'s.
\qed

\section{Positivity of limiting variances and the differences of $\eta$}
\label{sec7}\setcounter{equation}{0}

\subsection{The measures $\ka_{A_l}$}\label{Meas}

We say that the variables $b_{s,i}$ and $b_{t,j}$ are 
\emph{equivalent} if
there exist $x_{i,j}\in\bbZ$ and $z_{i,j}, z_{s,t}\in\bbQ$ such that 
$q_{r_i}(y)=q_{r_j}(y-x_{i,j})+z_{i,j}$
and $q_{r_s}(y)=q_{r_t}(y-x_{i,j})+z_{s,t}$, for any $y\in\bbR$. It is clear that this is an equivalence 
relation  
and notice that when such $x_{i,j},z_{i,j}$ and $z_{s,t}$ exist,  then $z_{i,j},z_{s,t}\in\bbZ$. 
 Denote by  $\cB$ the set of all equivalence classes. 
Let $B\in\mathcal B$ and write  $B=\{b_{s_k,i_k}: 1\leq k\leq n_B\}$, 
where  $s_k=s_k(B)$, $i_k=i_k(B)$. 
Then for any $1\leq k,j\leq n_B$, 
\begin{eqnarray}\label{z-prop}
 q_{r_{s_j}}(y)-q_{r_{s_k}}(y+x_{i_k,i_j})=z_{s_j,s_k}=z_{s_j,s_1}-z_{s_k,s_1} 
, \text{ for any }  y\in\bbR.
 \end{eqnarray}
 Let the measure 
$\ka^{(B)}$ be the law of the random  vector  $Y^{(B)}(0)$, where 
$
Y^{(B)}(n)=\{Y_{s_j}(n+R+{z_{s_j,s_1}}): 1\leq j\leq n_B\}
$,
$R=R(B)\in\bbZ$ satisfies $R\geq|z_{s_j,s_k}|$ for any  $1\leq j,k\leq n_B$ and 
\begin{eqnarray}\label{Ys}
Y_s(n)=(X(n+d_{1,s}),...,X(n+d_{r_s-r_{s-1},s}))
\end{eqnarray}
where $d_{j,s}=q_{r_{s-1}+j}-q_{r_{s-1}+1}$, which is a constant.

Next, let $A\in\cA$, fix some $1\leq l\leq m(A)$ and set $A_l=A_{m(A),l}$.
Let $q_{r_i}\in A_l$ and $1\leq s\leq i$. Then for any $j$ and $t$ 
the variables  $b_{t,j}$ and $b_{s,i}$ are equivalent only if 
$q_{r_j}\in A_l$. As a consequence, there exists
 $\cB(A_l)\subset\cB$ such that 
\begin{eqnarray}
&\{b_{s,i}: q_{r_i}\in A_l,\,1\leq s\leq i\}=\cup_{B\in\cB(A_l)}B.
\end{eqnarray}
 Let the measure $\ka_{A_l}$
be defined by 
$d\ka_{A_l}(b^{(A_l)})=\prod_{B\in\cB(A_l)}
d\ka^{(B)}((b_{s_k(B),i_k(B)})_{1\leq k\leq n_B})$  and the  measure 
$\ka_A$ be defined by 
$d\ka_A(b^{(A)})=\prod_{1\leq l\leq m(A)}d\ka_{A_l}(b^{(A_l)})$.

\begin{lemma}\label{thm2.6-}
Suppose that Assumption \ref{ass2.1} is satisfied. Let $A\in\cA$ consists of 
nonlinear polynomials. Then for any $q_{r_i},q_{r_j}\in A$ such that $i\leq j$ and 
 $c_{r_i,r_j}=1$, 
\begin{eqnarray}\label{Dij}
D_{i,j}=c_{r_i,i_A}\int F_i(b_i)F_j(b_j)d\ka_A(b^{(A)}).
\end{eqnarray}
As a consequence, for any $1\leq l\leq m(A)$, 
\begin{equation}\label{DA-form}
D_{A_l}^2=\int G_{A_l}^2(b^{(A_l)})d\ka_{A_l}(b^{(A_l)})
\end{equation}
and therefore $D_{A_l}^2=0$ if and only if $G_{A_l}$ vanishes $\ka_{A_l}$- 
almost surely.
\end{lemma}

\begin{proof}
Relying on Section \ref{sec3}, we assume without  loss of generality that  
$\hat\ell=\ell$, which means that $r_i=i$, $i=1,...,\ell$. 
 Let $1\leq i<j\leq\ell$ be such that 
$q_i$ and $q_j$ are nonlinear, equivalent and  $c_{i,j}=1$. 
Let $x_{i,j}\in\bbQ$ satisfying (\ref{5.0})
and $M=M(1,1,x_{i,j})$ be defined by (\ref{4.01}) with $u=2$ and $\al_2=\be_2=1$.
Then Remark \ref{rem4.2} shows that $M>0$ if and only if 
$x_{i,j}\in\bbZ$, and in this case  $M=1$. 
Thus,  by  Proposition \ref{prop5.3},
\begin{eqnarray}\label{DijForm}
\,\,\,\,\,\,&D_{i,j}=c_{j,i_{k-1}+1}\int F_i(x)F_j(y)dm_{i,j}(x,y)=c_{i,i_{k-1}+1}\int F_i(x)F_j(y)dm_{i,j}(x,y)
\end{eqnarray}
when  $x_{i,j}\in\bbZ$, while $D_{i,j}=0$ when $x_{i,j}\notin\bbZ$. 
Here $k$ is such that $\deg q_i=m_k$.

Next, let $q_i,q_j\in A$. It is clear that 
the variables $b_{s,i}$ and $b_{s',i}$ are not equivalent when $s\not=s'$. 
In particular the marginal of $\ka_A$ corresponding to the variable $b_i$
is $\mu^i=\mu\times\mu\times\cdots\times\mu$, and for any $B\in\cB$ the intersection
\begin{eqnarray}\label{Inters}
 B\cap(\{b_{s,i}: 1\leq s\leq i\}\cup\{b_{t,j}: 1\leq t\leq j\})
\end{eqnarray}
contains at most two variables. 
The components of $b_i$ and $b_j$  are clearly independent with respect to 
$\ka_A$ when $q_i,q_j\in A$, $c_{i,j}=1$ 
and  $x_{i,j}\not\in\bbZ$, and (\ref{Dij}) follows in this situation since both its sides
 vanish.
Next, suppose that  $i<j$, $c_{i,j}=1$ and $x_{i,j}\in\bbZ$ and consider the
(distinct) variables  $b_{s,i}$ and $b_{t,j}$. They are equivalent if and only if
(\ref{5.0+}) is satisfied, and in this case  by (\ref{z-prop})
the $\bbR^\wp\times\bbR^\wp$ marginal of $\ka_A$ corresponding to the pair $(b_{s,i},b_{t,j})$
is the measure $\mu_{q_t(0)-q_s(x_{i,j})}$. 
Hence,  the marginal of $\ka_A$ corresponding to the pair
$(b_i,b_j)$ is $m_{i,j}$,  and 
(\ref{Dij}) follows from (\ref{DijForm}).
Finally, (\ref{Dij}) when $i=j$ follows from  (\ref{5.1+}), since the marginal of $\ka_A$
corresponding to $b_i$ is $\mu^i$, and 
(\ref{DA-form}) clearly follows from (\ref{Dij}) and (\ref{2.19+}).
\end{proof}

\subsection{Proof of Theorem \ref{thm2.3}}

 Let $A\in\cA$ and write 
\begin{eqnarray*}
&A\cap\{q_{r_i}:1\leq i\leq\hat\ell\}=\{q_{r_{a_i}}: 1\leq i\leq s\}
\end{eqnarray*} 
where $a_1<a_2<...<a_s$.
Set $t_0=0$ and $t_i=c_{i_{k-1}+1,a_i}$, $i=1,...,s$, 
where $k$ is such that  $\deg q_{a_1}=m_k$.
Let   $L\in\bbN$, $j_1<...<j_{L-1}<s$ and $d_1<...<d_L$ be such that 
$t_i=d_l$ if $j_{l-1}< i\leq  j_l$, where  we set $j_L=s$ and $d_0=j_0=0$. 
Set $C_l=\{q_{r_{a_i}}:\, j_{l-1}<i\leq j_l\}$, $l=1,2,...,L$.  Then $A=\bigcup_{l=1}^LC_l$
and this is a disjoint union. Furthermore, 
each $A_{m(A),i}=A_i$, $1\leq i\leq m(A)$
is contained in some $C_l$, and  for each $l$ we have 
\begin{eqnarray}\label{C_l}
&D^2_{C_l}=d_l\lim_{N\to\infty}E[(\sum_{i: q_{r_i}\in C_l}\xi_{i,N}(1))^2]
\,\,\text{ and }\,\,C_l=\bigcup_{i: A_i\subset C_l}A_i
\end{eqnarray} 
 where this union is disjoint. 
For each $u=1,2...,L$ set
$C^{(u)}=\cup_{u\leq l\leq L}C_l=\{q_{r_{a_i}}:\, j_{u-1}<i\leq s\}$ and
$\cD_u=\lim_{N\to\infty}E[(\sum_{i: q_{r_i}\in C^{(u)}}\xi_{i,N}(1))^2]$ which can also be
written as
\begin{eqnarray*}
&\cD_u=\sum_{j_{u-1}<i\leq s}D_{a_i,a_i}+2\sum_{j_{u-1}<i<j\leq s}D_{a_i,a_j}.
\end{eqnarray*}
We first claim that 
\begin{equation}\label{2.21+}
D^2_A=\sum_{u=1}^L (d_u-d_{u-1})\cD_u.
\end{equation}
Indeed,  notice that
\begin{eqnarray}\label{I}
&D_A^2=\sum_{1\leq i\leq s}t_iD_{a_i,a_i}
+2\sum_{1\leq i<j\leq s}t_iD_{a_i,a_j}=\\
&d_1\cD_1+\sum_{j_1<i\leq s}(t_i-d_1)D_{a_i,a_i}+
2\sum_{j_1< i<j\leq s}(t_i-d_1)D_{a_i,a_j}\nonumber
\end{eqnarray}
since $t_i=d_1$ for any $ j_0=0<i\leq j_1$. 
Similarly, for any $u=2,...,L-1$,
\begin{eqnarray}\label{II}
&\sum_{j_{u-1}<i\leq s}(t_i-d_{u-1})D_{a_i,a_i}+
2\sum_{j_{u-1}< i<j\leq s}(t_i-d_{u-1})D_{a_i,a_j}=\\
&(d_u-d_{u-1})\cD_u+\sum_{j_u<i\leq s}(t_i-d_u)D_{i,i}+
2\sum_{j_u<i<j\leq s}(t_i-d_u)D_{a_i,a_j}.\nonumber
\end{eqnarray}
Formula (\ref{2.21+}) follows from (\ref{I}) and a repetitive use of (\ref{II})
with $u=2,...,L-1$, observing that  the sum of the last two sums from
(\ref{II}) equals $(d_L-d_{L-1})\cD_L$, when $u=L-1$.

Proving Theorem \ref{thm2.3}, recall that
$D_{i,j}=0$   when $c_{r_i,r_j}=1$  and $x_{r_i,r_j}\notin\bbZ$. Thus by 
 Lemma \ref{thm2.6-} we have 
\begin{eqnarray*}
&D_{C_l}^2=\sum_{i: A_i\subset C_l}D^2_{A_i}=\int G_{C_l}^2(b^{(C_l)})d\ka_A(b^{(C_l)})
\,\text{ for any }\, 1\leq l\leq L.
\end{eqnarray*}
 Hence, $D^2_{C_l}=0$ for any $1\leq l\leq L$ if and only if $D^2_{A_i}=0$ for any $1\leq i\leq m(A)$.
On the one hand,
suppose that 
$D_A^2=0$. Then by (\ref{2.21+}) we have $\cD_u=0$ for any $1\leq u\leq L$. By 
the first equality from (\ref{C_l}) we have $D_{C_L}^2=d_L\cD_L$. 
Therefore, $D_{C_L}^2=0$ and so by the Cauchy-Schwarz inequality
\begin{eqnarray*}
&\cD_u=\lim_{N\to\infty}E[(\sum_{i: q_{r_i}\in \cup_{l=u}^{L-1}C_l}\xi_{i,N}(1))^2] 
\,\,\,\text{ for any } 1\leq u\leq L-1
\end{eqnarray*} 
and in particular $\cD_{L-1}=d_{L-1}^{-1}D^2_{C_{L-1}}$, implying that
$D^2_{C_{L-1}}=0$. 
Proceeding this way with $u=1,...,L-1$ in place of $L$
 we see that $D^2_{C_l}=0$, for any $l=1,2,...,L$. On the other hand, suppose that 
$D^2_{C_l}=0$ for any $l$.  Then, by 
the first equality from (\ref{C_l}) and  by the Cauchy-Schwarz inequality 
 $\cD_u=0$ for any $u$, and therefore $D_A^2=0$
 by  (\ref{2.21+}).  
Finally, 
when $A$ consists of linear polynomials then $A_i$'s are singletons and so $D_A^2=0$
if and only if $D_{i,i}=0$ for any linear $q_{r_i}$. By the Cauchy-Schwartz inequality $D_{i,i}$ vanishes
for any linear $q_{r_i}$ if and only if $D_{i,j}$ vanishes for any linear $q_{r_i}$ and $q_{r_j}$,
where we took into account 
(\ref{2.19+}),  and the proof of Theorem
\ref{thm2.3}\emph{(i)} is complete. The proof of Theorem \ref{thm2.3}\emph{(ii)} is a direct 
consequence of  Theorem \ref{thm2.3}\emph{(i)} and Lemma \ref{thm2.6-}.   \qed

The following corollaries are immediate consequences of Theorem \ref{thm2.3}.

\begin{corollary}\label{cor2.3-}
Let $\xi_N^{(k)}$ be as defined before (\ref{2.19})  and 
 set $\tilde D^2=\lim_{N\to\infty}E\big(\sum_{k: m_k>1}\xi_N^{(k)}(1)\big)^2$.
Then $\tilde D^2=0$ if and only if $\sum_{i: \deg q_{r_i}>1}F_i(b_i)=0$,\,
for $\prod_{A\in\cA: d_A>1}\ka_A$ almost any $\{b_i: \deg q_{r_i}>1\}$.
\end{corollary}

\begin{corollary}\label{cor2.3+}
Let $A\in\cA$ consists of nonlinear polynomials. Suppose that for any  distinct $q_{r_i},q_{r_j}\in A$
there  exist no $l,z\in\bbZ$ such that $q_{r_i}(y)=q_{r_j}(y-z)+l$, for any $y\in\bbR$. 
Then  $D_A^2=0$ if and only if the functions $F_s$, $q_{r_s}\in A$ 
 vanish $\nu_1\times\cdots\times\nu_{\hat\ell}$-almost surely. As a consequence, 
if $q_1$ is nonlinear and for any distinct $q_{r_i},q_{r_j}$ there exists no such $l,z$, 
then $D^2=0$ if and only if $F$ vanishes $\nu_1\times\cdots\times\nu_{\hat\ell}$-almost surly.
\end{corollary}

\begin{remark} 
Let $p$ and $q$ be polynomials. Existence of $z,l\in\bbZ$ such that $q(y)=p(y-z)+l$
for any $y\in\bbR$ clearly forms an equivalence relation, which is finer than $\cA$.  The 
sets  $A_l$ are the classes of  the reduction of this relation
to $\{q_{r_1},...,q_{r_{\hat\ell}}\}$. Unlike for classes of $\cA$, the covariances 
$D_{i,j}$, $q_{r_i}\in A_l, q_{r_j}\in A_{l'}$ 
do not necessarily vanish when $A_l\not=A_{l'}$.  Still, 
Theorem \ref{thm2.3} shows that $D^2=0$ if and only 
$D_{A_l}^2=0$ for each $A_l$. 
\end{remark}

\subsection{Proof of Theorem \ref{thm2.4}}
Theorem \ref{thm2.4} follows from Theorem 2.3 in \cite{HK} in the case when $q_i(n)=in$
for any $n\in\bbN$ and $i=1,...,i_1$. The proof proceeds in the same 
way in the case when $q_i(n)=m_in+b_i$ for some natural numbers $m_1<m_2<...<m_{i_1}$ 
and  integers $b_1,...,b_{i_1}$
if we replace $i_1$ with $m_{i_1}$, considering now 
$N_{m_{i_1}}^{(j)}$ in place of $N_{i_1}^{(j)}$, $j\in\bbN$ (which are defined in
the proof from \cite{HK}) taking into account Lemma \ref{lem5.1}.
In Section \ref{sec3} we showed that
the problem can be reduced to the case when  $\hat\ell=\ell$,
i.e. $r_i=i$  and the leading coefficients of the linear polynomials satisfy
$a_1^{(1)}<a_1^{(2)}<...<a_1^{(i_1)}$,  recalling that in our situation
$q_i(n)=a_1^{(i)}n+a_0^{(i)}$. Since $q_i(n)\in\bbN$ for any $n\in\bbN$
we see that $a_1^{(i)},a_0^{(i)}\in\bbZ$ and $a_1^{(i)}\geq 1$. Theorem 
\ref{thm2.4} follows now by the described above modification of the proof of Theorem 2.3 
in  \cite{HK}.

\subsection{The increments of $\eta$}

We begin with the proof of Theorem \ref{thm2.7}.
Establishing (\ref{2.23}), let $0\leq t_1\leq t_2\leq t_3$.  
By  (\ref{2.19++}), 
\begin{eqnarray}\label{7.1-}
&\eta(t)=\sum_{A\in\mathcal A}\eta^{(A)}(t)\,\,\,\text{where }\,\,\,
\eta^{(A)}(t)=\sum_{s:q_{r_s}\in A}\eta(c_{i_A,r_s}t)
\end{eqnarray}
and $i_A=i_{k-1}+1$, where $k=k_A$ is such that $\deg q_i=m_k$ for any $q_i\in A$.
Since $\eta_s$ and $\eta_{s'}$ are independent if $q_{r_s}$ and $q_{r_{s'}}$ 
are not equivalent, we obtain by (\ref{2.19+}) that
\begin{eqnarray}\label{7.1}
&E[(\eta(t_3)-\eta(t_2))\eta(t_1)]=\sum_{A\in\mathcal A}
E[(\eta^{(A)}(t_3)-\eta^{(A)}(t_2))\eta^{(A)}(t_1)]=\\
&\sum_{A\in\mathcal A}
\sum_{s_1,s_2\in S(A)} T_{A,s_1,s_2}(t_3,t_2,t_1)D_{s_1,s_2}\nonumber
\end{eqnarray}
where $S(A)=\{1\leq s\leq\hat\ell:q_{r_s}\in A \}$ and
\begin{eqnarray*}
&T_{A,s_1,s_2}(t_3,t_2,t_1)=\min(c_{i_A,r_{s_2}}t_3,c_{i_A,r_{s_1}}t_1)-
\min(c_{i_A,r_{s_2}}t_2,c_{i_A,r_{s_1}}t_1).
\end{eqnarray*}
Next, suppose that $ t_3\leq Ct_1$ where
$C>1$ is defined by (\ref{2.22}).
Recall that $c_{i_A,r_s}$ is nondecreasing in $s$. Therefore,
$T_{A,s_1,s_2}(t_3,t_2,t_1)=0$ if  $s_2\geq s_1$ (since
$t_1\leq t_2\leq t_3$). On the other hand, if $s_1>s_2$ then
by the definition (\ref{2.22}) of $C$, using the inequality
$t_2\leq t_3\leq Ct_1$, we have
$T_{A,s_1,s_2}(t_3,t_2,t_1)=c_{i_A,r_{s_2}}t_3-c_{i_A,r_{s_2}}t_2$.
Hence by (\ref{7.1}),
\[
E[(\eta(t_3)-\eta(t_2))\eta(t_1)]=(t_3-t_2)
\sum_{A\in\mathcal A}\,\,\,
\sum_{s_1,s_2\in S(A):\,s_1>s_2} c_{i_A,r_{s_2}}D_{s_1,s_2}:=(t_3-t_2)\hat\Delta.
\]
By (\ref{2.20}),
$
D^2=\sum_{A\in\cA}\sum_{s: q_{r_s}\in A}
c_{i_A,r_s}D_{s,s}+2\hat\Delta
$
and therefore
$
\hat\Delta=\frac12\Delta
$, 
and (\ref{2.23}) follows.

Completing the proof of Theorem \ref{thm2.7}\emph{(i)}, let $K>0$, consider the
interval $I=[K,KC]\subset(0,\infty)$ and  let $t_0\leq t_1\leq t_2\leq t_3$
 in $I$. Then,  $t_3\leq KC\leq Ct_0\leq Ct_1$ and therefore by (\ref{2.23}),
\begin{eqnarray*}
&E\big[(\eta(t_3)-\eta(t_2))(\eta(t_1)-\eta(t_0))\big]
=\frac12(t_3-t_2)\Delta-\frac12(t_3-t_2)\Delta=0.
\end{eqnarray*}
Since $\eta$ is a Gaussian process this means that $\eta(t_3)-\eta(t_2)$ and
$\eta(t_1)-\eta(t_0)$ are independent and the first assertion from
Theorem \ref{thm2.7}$(i)$ follows. Proving the 
second assertion,  if $\Delta\not=0$ then  by (\ref{2.23}) the differences
$\eta(t_3)-\eta(t_2)$ and $\eta(t_1)-\eta(0)$ are not independent as
long as $0< t_1\leq t_2\leq t_3\leq Ct_1$. By continuity of the
covariances as functions of $t$, if $0<t_0$ is sufficiently small then also
$\eta(t_3)-\eta(t_2)$ and $\eta(t_1)-\eta(t_0)$ are not independent.
Proving Theorem \ref{thm2.7}\emph{(ii)}, we observe that the assumption there
implies that  $c_{i_A,r_s}$ is constant in $s$ on each set $\{ s:\, q_{r_s}\in A\}$, 
$A\in\mathcal A$. Thus, by the second equality from  (\ref{7.1}), for any 
$0\leq t_1\leq t_2\leq t_3$,
\begin{eqnarray*}
E[(\eta(t_3)-\eta(t_2))\eta(t_1)]=0
\end{eqnarray*}
which makes the increments independent.
In order to see that they are stationary
it is sufficient to show that for any $A\in\mathcal A$
the process $\eta^{(A)}$ defined in (\ref{7.1-})
has stationary increments, which holds true since the multidimensional process
$\{\eta_s,\,s=1,...,\hat\ell\}$ has stationary increments.\qed

Now we prove Corollary \ref{cor2.8}. Suppose that $\hat\ell=\ell=2$. 
If  $D_{1,2}=0$ then $\eta$ is a sum of two independent processes with independent
increments, and so it is  a process with independent increments. When 
$q_1\not\equiv q_2$ then $D_{1,2}=0$ by Proposition \ref{prop5.3}, 
and hence $\eta$ has independent increments.
On the other hand, suppose that $q_1\equiv q_2$.  If $c_{1,2}=1$ then Theorem 
\ref{thm2.7}\emph{(ii)} shows that $\eta$ has independent increments. 
When $c_{1,2}>1$ then Theorem \ref{thm2.7}\emph{(i)} shows that 
$\eta$ does not have independent increments if $D_{1,2}\not=0$, since in 
this situation $\Del$ and $D_{1,2}$ are proportional.

Finally, we give examples that $D_{1,2}$ may or may not vanish in the case 
when $q_1\equiv q_2$ and $\deg q_1>1$, 
no matter whether $c_{1,2}=1$ or $c_{1,2}>1$. Let $x_{1,2}$ satisfying  
(\ref{5.0}) and write  $c_{1,2}=\al/\be$, where $\gcd(\al,\be)=1$. 
Then by Remark \ref{rem4.2} and  Proposition \ref{prop5.3},
$D_{1,2}$ is proportional to
$
\int F_1(x)F_2(y,z)d\mu(y)d\mu_k(x,z)
$
where $k=q_2(0)-q_1(x_{1,2})$, assuming that $x_{1,2}=z-c_{2,1}t$ for some $z\in\bbZ$ and 
$t\in\{0,1,...,\al-1\}$. 
Suppose that $k=0$ and  consider functions of the form $F(x,y)=f_1(x)f_2(y)+g(x)$
where $\int g(x)d\mu(x)=\int f_2(y)d\mu(y)=0$.
Then $F_1=g$ and so $D_{1,2}$ is proportional to 
$
 \int f_1(y)d\mu(y)\times\int g(x)f_2(x)d\mu(x).
$
When $g=f_2$, $\int f_1(x)\mu(x)\not=0$ and 
$\int f^2_2(y)d\mu(y)>0$ then 
$D_{1,2}\not=0$, while $D_{1,2}=0$
if $\int f_1(y)d\mu(y)=0$. \qed

\subsection{Characterization of positivity for nonlinear classes via solutions for functional equations }
Set $\tilde A=\{q_{r_i}: \deg q_{r_i}>1\}$. 
Let $I\subset\tilde A$ and let $j_I$ be the maximal index $j$ such that $q_{r_j}\in I$.
Notice that $G_I(b^{(I)})=0$ if and only if the function 
$F_I(y_1,...,y_{j_I})=\sum_{i: q_{r_i}\in I}F_i(y_1,...,y_i)$ satisfies
\begin{eqnarray}\label{FCr}
&F_I(b_{j_I})=
\sum_{i<j_I: q_{r_i}\in I}F_i(p_i(b_{j_I}))-F_i(b_i)
\end{eqnarray}
where $p_i(z_1,...,z_{j_I})=(z_1,...,z_i)$. Let the measure $\ka_I$ be the marginal of 
$\tilde\ka=\prod_{A\in\cA: d_A>1}\ka_A$ corresponding to the variable $b^{(I)}$. 
 Consider the equation
\begin{eqnarray}\label{Eq-thm2.6-}
&F_I(b_{j_I})=
\sum_{i<j_I: q_{r_i}\in I}g_i(p_i(b_{j_I}))-g_i(b_i),\,\,
\ka_I-\text{almost surely}
\end{eqnarray}
where  $g_i$'s are functions satisfying (\ref{2.17}). When $I=\{q_{r_{j_I}}\}$
then (\ref{Eq-thm2.6-}) becomes $F_{j_I}=0$, and existence of such a solution means that 
$F_{j_I}$ vanishes $\nu_1\times\cdots\times\nu_{j_I}$-almost surely.

\begin{theorem}\label{thm-FunEq}
(i) The set of functions $\{F_i: i<j_I, q_{r_i}\in I\}$ is the only possible solution for 
the equation (\ref{Eq-thm2.6-}). 

(ii) Let  $\tilde D^2$ be as  in Corollary \ref{cor2.3-}. Then 
$\tilde D^2=0$ if and only if there exists a solution for the equation 
(\ref{Eq-thm2.6-}) with $I=\tilde A$.
 In particular when $q_1$ is nonlinear then $D^2=0$ if and only if there exists a solution 
$g=\{g_i: i<\hat\ell\}$ for the equation 
\begin{eqnarray*}
&F(b_{\hat\ell})=\sum_{i<\hat\ell}g_i(p_i(b_{\hat\ell}))-g_i(b_i),\,\,
\tilde\ka-\text{almost surely}
\end{eqnarray*}
with $g_i$'s satisfying (\ref{2.17}).
\end{theorem}
We note that  $\tilde D^2=0$
 is  equivalent to the statement that there exists a solution for (\ref{Eq-thm2.6-}) with either any 
$I=A\subset\tilde A$ or  any $I=A_l\subset\tilde A$,  as well.

\begin{proof}
 Let $I\subset\tilde A$, 
set $j=j_I$ and let $\{g_i: i<j, q_{r_i}\in I\}$  be a solution for  (\ref{Eq-thm2.6-}). 
 Let  $A,l$ be such that $q_{r_j}\in A_l$. 
The equivalence class $B_{j,j}$ containing the variable $b_{j,j}$ satisfies
$B_{j,j}=\{b_{i,i}: q_{r_i}\in A_l\}$ and 
the marginal of  $\tilde\ka$ corresponding to the variable $b_j$ is $\nu^{(j)}=
\nu_1\times\cdots\times\nu_j$.
 Thus, integration of both sides  of (\ref{Eq-thm2.6-}) with respect to  $\ka_{A'_{l'}}$ for any
 $A'_{l'}\not=A_l$ and then with respect to  $\ka^{(B_{j,j})}$ yields,
\begin{eqnarray}\label{UP}
&\sum_{i<j: q_i\in I}F_i(p_i(b_j))=
\sum_{i<j: q_i\in I}g_i(p_i(b_j)),\,\nu^{(j)}-
\text{almost surely}
\end{eqnarray}
where we used that $g_i$'s satisfy (\ref{2.17}). 
Completing the proof of Theorem \ref{thm-FunEq}\emph{(i)}, 
write  $I=\{q_{r_{k_1}},...,q_{r_{k_d}},q_{r_j}\}$ where $k_1<k_2<...<k_d<j$. 
 Integrating (\ref{UP})
with respect to $(b_{t,j})_{k_1<t\leq j}$, taking into account that $g_i$'s 
 satisfy (\ref{2.17}), 
 yields $g_{k_1}=F_{k_1}$, $\nu^{(j)}$-almost surely. Subtracting 
$F_{k_1}(p_{k_1}(b_j))$ from both sides of (\ref{UP}) and then 
 repeating this  argument with $k=k_2,k_3,...,k_d$ in place of $k_1$
 shows that  $g_i=F_i$,  $\nu^{(j)}$-almost surely, for any $i<j$ such that $q_{r_i}\in I$. 
Theorem \ref{thm-FunEq}\emph{(ii)} is a direct consequence of 
Theorem \ref{thm-FunEq}\emph{(i)} and Corollary (\ref{cor2.3-}). 
\end{proof}

\subsection{\textbf{The stationary case}}

Consider the situation when $X=\{X(n): n\geq 0\}$ is stationary.
Suppose that  $q_1$ is linear and let $k$ be such that $i_1=r_k$. 
Let $s^2$ be as in Theorem \ref{thm2.4}. Then 
 $D_{\cL_1}^2=\lim_{N\to\infty}E\big(\xi_N^{(1)}(1)\big)^2=0$ if and only if
$s^2=0$. 
The process $Y=\{(Y^{(i)}(n))_{i=1}^k:\,n\geq 0\}$ is stationary, as well. Let 
$(\cX,\la,V)$  be a measure preserving system (MPS)
and $\varphi$ be a vector valued function
 such that  $Y(n)=\varphi \circ V^n$ for any $n\geq 0$. 
Inequality  (4.2) from \cite{HK} is established in our situation with 
the function  $G=F_1+...+F_k$ in the same way as in \cite{HK}. Thus, 
by  Proposition 8.3 and Theorem 8.6 from \cite{Br}  (modified  for a one sided process)
$s^2=0$ if and only if 
the expectations $E[\Sigma_N^2]$ are bounded in $N$, which by Theorem 
18.2.2 from \cite{IL} (see also \cite{Bro}) is equivalent to existence of a square integrable function 
$g$ such that
\begin{eqnarray}\label{Cob1}
&G\circ\varphi=g-g\circ V,\,\,\,\la-\text{almost surely}.
\end{eqnarray}

 Similar equivalent condition for positivity  of  $D^2$  exists in the case that 
$q_\ell$ is nonlinear, as well. Indeed, the processes $Y^{(B)}=\{Y^{(B)}(n): n\geq 0\},\,B\in\cB$
defined above (\ref{Ys}) are stationary,
and let $M(B)=(\Om_B,\cM_B,U_B)$ be an MPS and $\phi_B$ be a vector valued function 
such that $Y^{(B)}(n)=\phi_B\circ U_B^n$ for any $n\geq0$. 
Set $\Om_{\cB}=\prod_{B\in\cB}\Om_B$ and $\cM_{\cB}=\prod_{B\in\cB}\cM_B$. 
For any  $1\leq i\leq \hat\ell$ and $1\leq s\leq i$
let $B_{s,i}\in\cB$ be such that $b_{s,i}\in B_{s,i}$. 
Let  the map $p_{b_i}$ be defined by 
$p_{b_i}(\om_{\cB})=(U_{B_{s,i}}(\om_{B_{s,i}}))_{s=1}^i$, where 
$\om_{\cB}=(\om_B)_{B\in\cB}\in\Om_{\cB}$. 
Set $\phi_{\hat\ell}=
\phi_{B_{1,\hat\ell}}\times\cdots\times\phi_{B_{\hat\ell,\hat\ell}}$.
The singletons 
$B_{i,\hat\ell}=\{b_{i,\hat\ell}\}$, $i=1,...,k$ are classes of $\cB$ 
since  $q_{r_i}$ is linear for such $i$'s.  Thus, 
the processes $Y^{(i)}$ and $Y^{(B_{i,\hat\ell})}$ have the same distribution.
Therefore $\prod_{s=1}^kM(B_{s,\hat\ell})$ (the product MPS) together with the 
function $\varphi=\phi_{B_{1,\hat\ell}}\times\cdots\times \phi_{B_{k,\hat\ell}}$
generate a process $\tilde Y$ which has the same distribution as the process $Y$
defined above. 

Suppose that  Assumption \ref{ass2.1} and (\ref{Mix2}) hold true. 
Combining the conditions  for positivity of $D_{\cL_1}^2$  and 
Theorem \ref{thm-FunEq} shows that  $D^2=0$ if and only if there exists a solution 
$(g, \{g_i: i<\hat\ell,\,\deg q_{r_i}>1\})$ for the equation   
\begin{equation}\label{Rep*}
F\circ\phi_{\hat\ell}\circ p_{b_{\hat\ell}}=
(g-g\circ V)\circ p_k\circ p_{b_{\hat\ell}} +
\sum _{k< i<\hat\ell}
g_i\circ p_{b_i}-g_i\circ p_i\circ p_{b_{\hat\ell}},\,\,\,\cM_{\cB}-a.s.
\end{equation}
where  a.s stands for almost surly.
Here  $V=U_{B_{1,\hat\ell}}\times\cdots\times U_{B_{k,\hat\ell}}$, 
 $p_i(z_1,...,z_j)=(z_1,...,z_i)$ for any $i\leq j$, $g$ is a square integrable 
function and $g_i$'s are functions satisfying
$
\int g_i((\om_{B_{s,i}})_{s=1}^i)d\cM_{B_{i,i}}(\om_{B_{i,i}})=0 
\text{ for any } \om_{B_{1,i}},...,\om_{B_{i-1,i}}.
$
When $q_1$ is nonlinear then  the term $g-g\circ V$ does not appear, 
 we set $k=0$  and we only require 
 Assumption \ref{ass2.1} to be satisfied. 
We can always assume that $U_B$'s are invertible and then to  replace 
$U_{B_{s,\hat\ell}}(\om_{B_{s,\hat\ell}})$ with $\om_{B_{s,\hat\ell}}$. 
In this case the left hand side becomes $F\circ\phi_{\hat\ell}$
and the first term on the right hand side becomes $g-g\circ V$.

Let $(\Om,\cM,U)$ be an MPS and  $\phi$ be a vector valued function such that $X(n)=\phi\circ U^n$
for any $n\geq0$. 
We can always take the natural MPS of $Y^{(B)}$. This means that
$\Om_B=\prod_{b_{i,s}\in B}\Om^{r_s-r_{s-1}}$,
$\cM_B=\bbI_{\textbf{D}_B}\cM$ and 
$U_B=\prod_{b_{i,s}\in B}\prod_{1\leq j\leq r_s-r_{s-1}}U$. 
Here \textbf {D}$_B$  is the diagonal sets of $\Om_B$ and $\bbI_{\textbf{D}_B}$
is its  indicator function.  Then
$Y^{(B)}(n)=\phi_B\circ U_B^n$, where
$
\phi_B=
\phi\circ\prod_{b_{i,s}\in B}\prod_{1\leq j\leq r_s-r_{s-1}} U^{R(B)+j_{b{i,s}}+d_{j,s}}
$
and, in the notations appearing above (\ref{Ys}),
 $j_b=z_{s_k(B),s_1(B)}$ if $b=b_{s_k(B),i_k(B)}$.
In this case the equation (\ref{Rep*}) includes functions and powers of $U$, which makes it 
explicit in  terms of $U$ and $\phi$.

\section{{\large Appendix: General weak limit theorems}}
\label{sec8}\setcounter{equation}{0}

For each $N\in\bbN$ let $\cG_{N,n},\,n=1,2,...$ be a filtration of $\sig$-algebras
and let $\{U_{N,n}:n\geq1\}$ be a triangular array of random variables
satisfying the following conditions:
\begin{description}
\item[B1]
For any $N$, $\{U_{N,n}:n\geq1\}$ is adapted to some
 $\big(\Om_N,\cG_{N,n},P_N\big),\,\, n=1,2,3...$;
\item[B2]
$\{U_{N,n}\}$ are uniformly square integrable;
\item[B3]
$\|E[U_{N,m}|\cG_{N,n}\|_2\leq c(m-n)$ for all $N$, $n\leq m$ and some sequence
$c(k)$ satisfying $\sum_{k=0}^\infty c(k)=C<\infty$;
\item[B4]
For some increasing function $A(t)$,
\[
\lim_{N\to\infty}\|\frac1N\sum_{1\leq n\leq Nt}W^2_{N,n}-A(t)\|_1=0
\]
where
\begin{equation}\label{W-def}
W_{N,n}=U_{N,n}+\sum_{m\geq n+1}E(U_{N,m}|\cG_{N,n})-
\sum_{m\geq n}E(U_{N,m}|\cG_{N,n-1}).
\end{equation}
\end{description}
Observe that for any fixed $N$ the process $\{W_{N,n},n\geq1\}$ is
a martingale difference sequence with respect to $\{\cG_{N,n}, n\geq 1\}$,
 provided that conditions B1-B3 hold true and that
condition B4 is a usual quadratic variation type condition.
The following theorem is a standard result cited in \cite{KV} as Theorem 5.1
(see, for instance, \cite{JS}).
\begin{theorem}\label{Thm5.1} Suppose that conditions \emph{B1-B4} are 
satisfied. Then, for any $T>0$ the processes
\[
\zeta_N(t)=\sum_{1\leq n\leq Nt}U_{N,n}
\]
converge in distribution on the Skorokhod space $D\big([0,T], \bbR\big)$
 to a Gaussian process
$\zeta(t)$ with independent increments such that $\zeta(t)-\zeta(s)$ has
mean $0$ and variance $A(t)-A(s)$.
\end{theorem}
In Lemma 5.2 and Remark 5.3 from \cite{KV} it is explained that
condition $B4$ can be replaced by the weaker condition
\begin{eqnarray}\label{B4'}
\lim_{N\to\infty}\frac1N\sum_{1\leq n\leq Nt}EW^2_{N,n}=A(t)
\end{eqnarray}
if condition B2 is satisfied and one can write
 $U_{N,n}=H_n\big(X_r(q_1(n)),X_r(q_2(n)),...,X_r(q_{j-1}(n)),\om\big)$,
 where $X_r(n)=E(X(n)|\cF_{n-r,n+r})$.
Here $r$ is a constant independent of $n$ and $N$ and $H_n(x_1,...,x_{j-1},\om)$ is
$\cF_{q_j(n)-r,q_j(n)+r}-$measurable  such that
 $||H(x,\cdot)||_2\leq K(1+||x||^\iota)$ for any $x=(x_1,...,x_{j-1})$.
 This remains true also in our polynomial setup since, after the reduction to the 
case $\hat\ell=\ell$, the differences
$q_i(n)-q_{i-1}(n)$
 and $q_i(n+1)-q_i(n)$ grow at least as fast as linear which makes  Lemma 5.2
from \cite{KV} applicable.

The following proposition is proved in  \cite{KV} (see Corollary 5.7 from there).
\begin{proposition}\label{Cor5.7} 
Assume that we have a triangular array consisting of $\cG_{N,n}$-measurable
random vectors $U_{N,n}=(U_{N,n}^{(i)})_{i=1}^d:\Om\to\bbR^d$ and that any linear
combination $\langle\la,U_{N,n}\rangle$ satisfies conditions \emph{B1-B4}. 
In particular,
\[
\lim_{N\to\infty}\big\|\frac1N\sum_{1\leq n\leq Nt}
\langle\la,W_{N,n}\rangle^2-\langle\la,A(t)\la\rangle\big\|_1=0
\]
where $W_{N,n}=(W_{N,n}^{(i)})_{i=1}^d$ and  $W_{N,n}^{(i)}$
 is defined by (\ref{W-def}) with  the process $\{U_{N,n}^{(i)}, n\geq1\}$.  
Let $\{k_N\}$ satisfying  $\lim_{N\to\infty}\frac{k_N}N=t_0$. Then
for any $T>0$,
\[
\zeta_{N,k_N}(t)=\frac1{\sqrt N}\sum_{k_N+1\leq n\leq k_N+Nt}U_{N,n}
\]
converges in distribution on the Skorokhod space $D\big([0,T];\bbR^d\big)$
to a Gaussian process $\eta(t)$ with independent increments taking values in
$\bbR^d$, having mean $0$ and covariances
\[
E\langle\la,\eta(t)-\eta(s)\rangle^2=\langle\la,\big(A(t)-A(s)\big)\la\rangle.
 \]
\end{proposition}
Next, 
\begin{theorem}\label{Thm5.6}
 Let $\{U_{N,n},\,n\geq1\}$, $\{k_N\}$ and $\zeta_{N,k_N}(t)$
be as in Proposition \ref{Cor5.7}.
 Let $\mathcal X$ be a complete metric space and for each $N\geq1$
 let $F_N(\om)$ be a $\mathcal X$ valued and $\cG_{N,k_N}-$ measurable
 random variable. Suppose that the distribution $\gam_N$ of $F_N$ under
 $P_N$ converges weakly as $N\to\infty$ to $\gam$ on $\mathcal X$.
 Then for any $T>0$ the joint distribution of the pair $\big(F_N, \zeta_{N,k_N}(\cdot)\big)$
 converges on $\mathcal X\times D\big([0,T];\bbR^d\big)$
 to the product of $\gam$ and the
 distribution of a Gaussian process with independent increments having mean
 $0$ and a covariance matrix $A(t+t_0)-A(t_0)$. 
We can drop the assumption that $\frac{k_N}N\to t_0$ provided that 
\[
\lim_{N\to\infty}\big\|\frac1N\sum_{k_N+1\leq n\leq k_N+Nt}
\langle\la,W_{N,n}\rangle^2-\langle\la,A(t)\la\rangle\big\|_1=0
\]
for any $t\geq 0$ and $\la\in\bbR^d$. 
\end{theorem}
This result was proved in Theorem 5.6 from \cite{KV} for one-dimensional
processes $\{U_{N,n},\,n\geq 1\}$ and the proof of the multidimensional 
version above proceeds in a similar way, relying  on Proposition \ref{Cor5.7}
in place of  Theorem \ref{Thm5.1}

\begin{remark}\label{rem6.1}
In  Theorem \ref{Thm5.1}, Proposition \ref{Cor5.7} and Theorem
\ref{Thm5.6} it is possible to replace $N$ by any monotone
increasing subsequence $\{N_j,j\geq1\}\subset\bbN$, i.e. to assume
that all the conditions are valid along this subsequence considering only 
$\zeta_{N_j}$ and $\zeta_{k_{N_j},N_j}$ and taking all the limits as $j\to\infty$.
\end{remark}

\end{document}